\documentclass[reqno,13pt, a4paper]{amsart}

\usepackage[a4paper,left=25mm,right=25mm,top=30mm,bottom=30mm,marginpar=25mm]{geometry} 
\usepackage{amsmath}
\usepackage{verbatim}
\usepackage{amssymb}
\usepackage{amsthm}
\usepackage{amscd}
\usepackage[ansinew]{inputenc}
\usepackage{cite}
\usepackage{bbm}
\usepackage{color}
\usepackage[english=american]{csquotes}
\usepackage[final]{graphicx}

\usepackage[plainpages=false,pdfpagelabels,backref=page,citecolor=red]{hyperref}
\usepackage{xcolor}
\definecolor{bole}{rgb}{0.47, 0.27, 0.23}
\definecolor{burgundy}{rgb}{0.5, 0.0, 0.13}
\graphicspath{{Pictures/}}

\numberwithin{equation}{section}

\newtheoremstyle{thmlemcorr}{10pt}{10pt}{\itshape}{}{\bfseries}{.}{10pt}{{\thmname{#1}\thmnumber{ #2}\thmnote{ (#3)}}}
\newtheoremstyle{thmlemcorr*}{10pt}{10pt}{\itshape}{}{\bfseries}{.}\newline{{\thmname{#1}\thmnumber{ #2}\thmnote{ (#3)}}}
\newtheoremstyle{defi}{10pt}{10pt}{\itshape}{}{\bfseries}{.}{10pt}{{\thmname{#1}\thmnumber{ #2}\thmnote{ (#3)}}}
\newtheoremstyle{remexample}{10pt}{10pt}{}{}{\bfseries}{.}{10pt}{{\thmname{#1}\thmnumber{ #2}\thmnote{ (#3)}}}
\newtheoremstyle{ass}{10pt}{10pt}{}{}{\bfseries}{.}{10pt}{{\thmname{#1}\thmnumber{ A#2}\thmnote{ (#3)}}}

\theoremstyle{thmlemcorr}
\newtheorem{theorem}{Theorem}
\numberwithin{theorem}{section}
\newtheorem{lemma}[theorem]{Lemma}

\newtheorem{proposition}[theorem]{Proposition}

\theoremstyle{thmlemcorr*}
\newtheorem{theorem*}{Theorem}
\newtheorem{lemma*}[theorem]{Lemma}
\newtheorem{corollary*}[theorem]{Corollary}
\newtheorem{proposition*}[theorem]{Proposition}
\newtheorem{problem*}[theorem]{Problem}
\newtheorem{conjecture*}[theorem]{Conjecture}

\theoremstyle{defi}
\newtheorem{definition}[theorem]{Definition}

\theoremstyle{remexample}
\newtheorem{remark}[theorem]{Remark}
\newtheorem{example}[theorem]{Example}

\theoremstyle{ass}

\newcommand{\Ical}{\mathcal{I}}

\newcommand{\Mcal}{\mathcal{M}}
\newcommand{\Ncal}{\mathcal{N}}

\DeclareMathOperator{\dist}{dist}

\newcommand{\norm}[1]{\|#1\|}

\newcommand{\normb}[1]{\bigl\|#1\bigr\|}

\newcommand{\normBB}[1]{\biggl\|#1\biggr\|}

\newcommand{\abs}[1]{|#1|}

\newcommand{\absb}[1]{\bigl|#1\bigr|}

\newcommand{\dd}{\;\mathrm{d}}

\newcommand{\R}{\mathbb{R}}

\newcommand{\loc}{\mathrm{loc}}

\newcommand{\sbullet}{\begin{picture}(1,1)(-0.5,-2)\circle*{2}\end{picture}}
\newcommand{\frarg}{\,\sbullet\,}

\def\Xint#1{\mathchoice 
{\XXint\displaystyle\textstyle{#1}}%
{\XXint\textstyle\scriptstyle{#1}}%
{\XXint\scriptstyle\scriptscriptstyle{#1}}%
{\XXint\scriptscriptstyle\scriptscriptstyle{#1}}%
\!\int} 
\def\XXint#1#2#3{{\setbox0=\hbox{$#1{#2#3}{\int}$} 
\vcenter{\hbox{$#2#3$}}\kern-.5\wd0}} 
\def\dashint{\,\Xint-}

\numberwithin{equation}{section}

\newcommand{\e}{\varepsilon}

\renewcommand{\epsilon}{\varepsilon}
\renewcommand{\phi}{\varphi}
\newcommand{\note}[1]{\strut{\color{red}#1}}

\title{Higher H\"older regularity for a subquadratic nonlocal parabolic equation}
\author{Prashanta Garain, Erik Lindgren and Alireza Tavakoli}
\date{\today}

\begin{document}

\maketitle

\begin{abstract}\noindent
In this paper, we are concerned with the H\"older regularity for solutions of the nonlocal evolutionary equation
$$
\partial_t u+(-\Delta_p)^s u = 0.
$$ 
Here, $(-\Delta_p)^s$ is the fractional $p$-Laplacian, $0<s<1$ and  $1<p<2$. We establish H\"older regularity with explicit H\"older exponents. We also include the inhomogeneous equation with a bounded inhomogeneity. In some cases, the obtained H\"older exponents are almost sharp. 
Our results complement the previous results for the superquadratic case when $p\geq 2$.
\end{abstract}

\tableofcontents

\section{Introduction and main result}

In this article, we establish higher H\"older continuity of weak solutions for the subquadratic nonlocal parabolic equation
\begin{equation}\label{meqn}
\partial_t u+(-\Delta_p)^s u=0.
\end{equation}
Here $1<p<2$ and $0<s<1$. The operator $(-\Delta_p)^s$ is the fractional $p$-Laplace operator defined by
$$
(-\Delta_p)^s u(x,t)=2\,\text{P.V.}\int_{\mathbb{R}^N}\frac{|u(x,t)-u(y,t)|^{p-2}(u(x,t)-u(y,t))}{|x-y|^{N+ps}}\,dy,
$$
where P.V. denotes the principal value.

The main contribution of this paper is to prove that solutions are almost $\Theta$-regular in space and almost $\Gamma$-regular in time,
where 
$$
\Theta(s,p):=\left\{\begin{array}{rl}
\dfrac{s\,p}{p-1},& \mbox{ if } s<\dfrac{p-1}{p},\\
&\\
1,& \mbox{ if } s\ge \dfrac{p-1}{p},
\end{array}
\right.\quad \mbox{ and }\quad \Gamma(s,p):=\left\{\begin{array}{rl}
1,& \mbox{ if } s<\dfrac{p-1}{p},\\
&\\
\dfrac{1}{s\,p-(p-2)},& \mbox{ if } s\ge \dfrac{p-1}{p}.
\end{array}
\right.
$$
We also allow a bounded right hand side in the equation, in which case these exponents are almost sharp. See Section \ref{sec:comment}.

These results are complementing the already existing H\"older estimates in \cite{BLS}, valid for $p\geq 2$.

\subsection{Main result}
Below we state our main theorem. For notation and relevant definitions, see Section \ref{sec:prel}.
\begin{theorem}
\label{teo:1}
Let $\Omega\subset\mathbb{R}^N$ be a bounded and open set, $I=(t_0,t_1]$, $1<p< 2$ and $0<s<1$. Suppose $(x_0,T_0)\in \Omega\times I$ is such that 
\[
Q_{2R,(2R)^{s\,p}}(x_0,T_0)\Subset\Omega\times I.
\]
Assume that $u\in L^p(J;W^{s,p}_{\mathrm{loc}}(\Omega))\cap L^{p-1}(J;L^{p-1}_{sp}(\mathbb{R}^N))\cap C(J;L^2_{\mathrm{loc}}(\Omega))$ is a local weak solution of 
\[
\partial_t u+(-\Delta_p)^s u=f\qquad \mbox{ in }\Omega\times I,
\]
as in Definition \ref{subsupsolution} where $f \in L^\infty_{\loc}(\Omega \times I)$  and
$$
\norm{u}_{L^\infty(Q_{2R,(2R)^{s\,p}}(x_0,T_0))}+\sup_{-(2R)^{sp}+T_0<t\leq T_0} \mathrm{Tail}_{p-1,s\,p}(u(\frarg,t);x_0,R)   <\infty.
$$
Define the exponents
\begin{equation}
\label{exponents}
\Theta(s,p):=\left\{\begin{array}{rl}
\dfrac{s\,p}{p-1},& \mbox{ if } s<\dfrac{p-1}{p},\\
&\\
1,& \mbox{ if } s\ge \dfrac{p-1}{p},
\end{array}
\right.\quad \mbox{ and }\quad \Gamma(s,p):=\left\{\begin{array}{rl}
1,& \mbox{ if } s<\dfrac{p-1}{p},\\
&\\
\dfrac{1}{s\,p-(p-2)},& \mbox{ if } s\ge \dfrac{p-1}{p}.
\end{array}
\right.
\end{equation}
Then\[
u\in C^\theta_{x,\rm loc}(\Omega\times I)\cap C^\gamma_{t,\rm loc}(\Omega\times I),\qquad \mbox{ for every }0<\theta<\Theta(s,p) \ \mbox{ and } \ 0<\gamma<\Gamma(s,p).
\] 
\par
More precisely, for every $(x_0,T_0)$ such that 
\[
Q_{2R,(2R)^{s\,p}}(x_0,T_0)\Subset\Omega\times I,
\] 
there exists a constant $C=C(N,s,p,\theta,\gamma)>0$ such that
\begin{equation}
\label{eq:xest}
\sup_{t\in \left[T_0-(R/4)^{s\,p},T_0\right]} [u(\cdot,t)]_{C^{\theta}(B_{R/4}(x_0))}\leq CM^{1+\frac{2-p}{sp}}R^{-\theta}
\end{equation}
and
\begin{equation}
\label{eq:test}
\sup_{x\in B_\frac{R}{4}(x_0)}[u(x,\cdot)]_{C^\gamma((-4^{-sp}R^{sp}+T_0,T_0])}\leq C MR^{-\gamma},
\end{equation}
where
\[
M=\|u\|_{L^\infty(Q_{R,R^{sp}}(x_0,T_0))}+\sup_{t\in (-R^{sp}+T_0,T_0]}\mathrm{Tail}_{p-1,s\,p}(u(\cdot,t);x_0,R)^{p-1} + R^{sp}\norm{f}_{L^\infty(Q_{R,R^{sp}}(x_0,T_0))} +1.
\]
\end{theorem}

\subsection{Comments on our result}\label{sec:comment}
Regarding the almost sharpness of the spatial H\"older regularity, it suffices to study the stationary problem. Indeed, see in Example 1.5 in \cite{BrLiSc} and Section 1.2 in \cite{GL}, where it is proved that if the inhomogeneity is merely in $L^q$, then solutions may not be $C^{sp/(p-1)}$ in general. See also Example 1.6 in \cite{BrLiSc}, where it is shown that for $p=2$ and $s=1/2$, solutions are not in general Lipschitz in space.

We point out that the time regularity is almost sharp in the case $s p\leq (p-1)$. Indeed, in Remark 1.3 in \cite{BLS} it is proved that solutions are not better than Lipschitz in time in general.

Finally, we wish to comment on the assumption on the tail. This is relevant, since there is a recent development on the regularity theory under the weaker assumption that the tail is merely in $L^q$ for some $q>p-1$ (cf. \cite{BK24}, \cite{Kassarxiv} and \cite{Lia24}), while we in the present paper assume that the tail is in $L^\infty$. This assumption is not unreasonable. Indeed, consider the following modified version of Example 5.2 in \cite{Kassarxiv}.
\begin{example}
Let $0<s<1$ and $1<p<2$ be such that $sp\leq p-1$. Suppose $u_n$ is a weak solution of
\[
 \left\{\begin{array}{rcll}
\partial_tu_n + (-\Delta_p)^su_n&=&0,&\mbox{ in } B_1\times (-1,1],\\
u_n&=&g_n,& \mbox{ on }(\mathbb{R}^N\setminus B_1)\times (-1,1],\\
u_n(\cdot,-1) &=& 0,&\mbox{ on } B_1.
\end{array}\right.
\]
Here, $g_n(x,t)=0$ when $x\in\mathbb{R}^N$ and $t\in (-1,0]$ while for $x\in\mathbb{R}^N$ and $t\in (0,1]$,
$$
g_n(x,t)=\delta f_n(t)+f_n'(t)\chi_{B_3\setminus B_2}(x),
$$
where we  for $n>e$ define
$$
f_n(t)=\begin{cases}
e^{-1},\quad e^{-1}\leq t\leq 1,\\
 -t\ln t, \quad \frac{1}{n}<t<e^{-1},\\
 t\ln n, \quad 0<t\leq\frac{1}{n}.
 \end{cases}
$$
In addition, $\delta >0$ is chosen so that 
$$
\delta + (-\Delta_p)^s \chi_{B_3\setminus B_2}(x)\leq 0,
$$
for all $x\in \R^N$.  
The existence of such a solution $u_n$ is guaranteed by for instance Theorem 1.4 in \cite{GKPT22}. To this end, we note that since $sp\leq (p-1)  <1$, $g_n$ lies in $W^{s,p}(\R^N)$ and with this choice of $f_n$ we have  $f_n,f_n'\in L^\infty(-1,1]$. Hence, $g_n$ is admissible for this theorem. For $x\in B_1$ and $t\in (-1,0]$, we have $\partial_t g_n+(-\Delta_p)^s g_n=0$ since $g_n(x,t)=0$ for $t\in (-1,0]$. We also see that $f_n'(t)\geq 0$ and therefore, for $x\in B_1$ and $t\in(0,1]$, we have
$$
\partial_t g_n+(-\Delta_p)^s g_n = \delta f_n'(t)+f_n'(t)(-\Delta_p)^s \chi_{B_3\setminus B_2}(x)\leq 0.
$$
Hence, $g_n$ is a subsolution in $B_1\times (-1,1)$.  By the comparison principle (see for example Proposition A.3 in \cite{LiLi23}), $u_n\geq g_n$ for $x\in B_1$ and $t\in (-1,1]$ and also $u_n(x,t)=0$\,(note that $g_n(x,t)=0$ for $x\in \R^N$ and $t\in (-1,0]$) for all $x\in B_1$ and $t\in (-1,0]$. This implies that for $x\in B_1$, 
$$
\frac{u_n(x,2/n)-u_n(x,-2/n)}{2/n}\geq \frac{g_n(2/n)-g_n(0)}{2/n}=  \frac{g_n(2/n)}{2/n}=\ln n-\ln 2\to \infty\text{ as }n\to\infty.
$$
Hence, the sequence $\{u_n\}$ is not uniformly Lipschitz in time while $$\mathrm{Tail}_{p-1,s\,p}(u_n(\cdot,t);0,1)^{p-1}\simeq f_n^{p-1}(t)+(f_n')^{p-1}(t)$$ is uniformly bounded in $L^q(-1,1)$ for all $q  <\infty$ but not for $q=\infty$. This shows that we cannot have an apriori Lipschitz estimate in time, in terms of the $L^q$ norm for any finite $q$. Hence, in some sense the $L^\infty$-assumpiton on the tail is almost optimal.

\end{example}
By instead choosing $f_n(t)\sim t^{1-\e}$ one can show for each $q$ there is an upper bound on which H\"older seminorm (in time) can be bounded in terms of the $L^q$ norm of the tail.

\subsection{Known results}
The fractional $p$-Laplace equation
\begin{equation}\label{fracp}
(-\Delta_p)^s u=0,\quad 1<p<\infty,
\end{equation}
has been widely studied and the available literature is vast. For $p=2$, see \cite{Kasscvpde} for a  local H\"older continuity result. We also refer to \cite{Silvestre} and the references therein.

In the nonlinear case $1<p<\infty$, H\"older regularity and Harnack inequalities for weak solutions are proved in \cite{Kuusi1,Kuusi2,Cozzi}. In \cite{Erik16}, H\"older continuity for viscosity solutions is established. H\"older continuity up to the boundary and fine boundary regularity has been studied in \cite{Antoni1, Antoni2,Antoni3,Antoni4}

Higher regularity results for the nonlocal equation \eqref{fracp} are established  in \cite{BrLiSc} for the case $p\geq 2$ and in \cite{GL} for $p<2$ respectively. 

Recently, higher H\"older regularity in the case $p=2$ with a more general kernel is established in \cite{Nowakcvpde}. We also mention \cite{KMS,DS23}, where sharp regularity results for the equation with a right hand side belonging to a Lorentz space are studied.

The parabolic counterpart of the equation \eqref{fracp}, namely equation \eqref{meqn}, has recently been studied by various authors.

When $p=2$, local H\"older continuity of weak solutions for the equation \eqref{meqn} is proved in \cite{KassFel, KassSch, Caf11}. A Harnack  inequality and local boundedness results can  for instance be found  in \cite{Martinanp, Kassarxiv, Kimna, Kimjde, BKK23}. In addition, higher Sobolev regularity results can be found in \cite{Zuazua18}.

In the nonlinear parabolic case, when $p>2$, local boundedness result for the equation \eqref{meqn} is proved in \cite{Martinjde} with zero right-hand side. For the non-zero right-hand side, local boundedness results can be found in \cite{Zhangcvpde} for any $1<p<\infty$ and in \cite{T} for any $p\geq 2$ respectively. Local H\"older continuity is proved in \cite{Zhangcvpde} for $p>2$ and in \cite{Naian, Adi} for any $1<p<\infty$ respectively.

For higher H\"older regularity results, when $p\geq 2$, we refer to \cite{BLS} with zero right hand side and \cite{GKS23,T} for non-zero right hand side respectively. Recently, higher H\"older regularity was obtained in the case $p=2$ allowing for more general kernel and a right hand side, in \cite{BKK23}.

It is worth mentioning that most of the results for nonlocal parabolic equations have been obtained under the assumption that the tail is locally bounded in time. Recently, there has been some advances when less restrictive assumptions on the tail have been assumed. This was initiated in \cite{Kassarxiv} for the linear setting and considered for the case of the fractional $p$-Laplacian in \cite{BK24} and \cite{Lia24}.

\textbf{Plan of the paper:} In Section 2, we introduce some notation and discuss preliminary results used throughout the paper. This is followed by Section 3, where we obtain an improvement of regularity on the Besov scale, given some initial H\"older regularity. This is then iterated and yields the desired spatial regularity in the normalized setting in Section 4. In Section 5, the time regularity is obtained, using the known spatial regularity. Finally, in Section 6, we prove our main result by combining the previous results.
\\

\section{Preliminaries}\label{sec:prel} 
In this section, we introduce some notation and present preliminary results needed in the rest of the paper. Throughout the paper, we shall use the following notation: $B_r(x_0)$ denotes the ball of radius $r$ with center at $x_0$. When $x_0=0$, we write $B_r(0):=B_r$.  Its Lebesgue measure
is given by 
$$
|B_r(x_0)|=\omega(N)r^N,
$$
where $\omega(N)$ is the volume of the unit ball $B_1$. We use the following notation for the parabolic cylinder
$$
Q_{R,r}(x_0,t_0)=B_{R}(x_0)\times(t_0-r,t_0].
$$
Again, when $x_0=0$ and $t_0=0$, we simply write $Q_{R,r}$.
The conjugate exponent $\frac{l}{l-1}$ of $l>1$ will be denoted by $l'$. We write
$c$ or $C$ to denote a positive constant which may vary from line to line or even in
the same line. The dependencies on parameters are written in the parentheses.

For any $1<q<\infty$, we define the monotone function $J_q:\mathbb{R}\to\mathbb{R}$ by
\begin{equation}\label{jp}
J_q(t)=|t|^{q-2}t
\end{equation}
and for $0<s<1$ and $1<p<\infty$ we use the notation
\begin{equation}\label{dmu}
d\mu=\frac{\dd x \dd y}{|x-y|^{N+ps}}.
\end{equation}

If $\psi:\mathbb{R}^N\to\mathbb{R}$ is a function and $h\in\mathbb{R}^N$, then  for any $x\in\mathbb{R}^N$, we define
\[
\psi_h(x)=\psi(x+h),\qquad \delta_h \psi(x)=\psi_h(x)-\psi(x)\\
\]
and
\begin{equation}\label{alpha0}
 \delta^2_h \psi(x)=\delta_h(\delta_h \psi(x))=\psi_{2\,h}(x)+\psi(x)-2\,\psi_h(x).
\end{equation}
Let $h\in\mathbb{R}^N$ and $\phi,\psi:\mathbb{R}^N\to\mathbb{R}$ be two given functions, then the following discrete product rule holds:
\begin{equation}\label{Lrule1}
\delta_h(\phi\psi)(x)=(\psi_h\delta_h\phi+\phi\delta_h\psi)(x),\quad \forall x\in\mathbb{R}^N.
\end{equation}
Similarly, for $\phi:\mathbb{R}^N\times I\to\mathbb{R}$, for any fixed $t\in I$ and for every $x\in\mathbb{R}^N$, we define
\[
\phi_h(x,t)=\phi(x+h,t),\qquad \delta_h \phi(x,t)=\phi_h(x,t)-\phi(x,t)\\
\]
and
\begin{equation}\label{alphabeta}
 \delta^2_h \phi(x,t)=\delta_h(\delta_h \phi(x,t))=\phi_{2\,h}(x,t)+\phi(x,t)-2\,\phi_h(x,t).
\end{equation}
\textbf{Function spaces:} Let $E\subset\mathbb{R}^N$, $N\geq 1$ be an open and bounded subset.
Suppose $v:E\to\mathbb{R}$ is a function. Then for {$0\leq \delta\leq 1$, we define the $\delta$-H\"older seminorm by
$$
[v]_{C^\delta(E)}:=\sup_{x\neq y\in E}\frac{|v(x)-v(y)|}{|x-y|^{\delta}}.
$$

It will be necessary to introduce two function spaces. For this reason, let $1\le q<\infty$ and take  $\psi\in L^q(\mathbb{R}^N)$. For $0<\beta\le 1$, define
\[
[\psi]_{\mathcal{N}^{\beta,q}_\infty(\mathbb{R}^N)}:=\sup_{|h|>0} \left\|\frac{\delta_h \psi}{|h|^{\beta}}\right\|_{L^q(\mathbb{R}^N)}.
\]
The Besov-type spaces $\mathcal{N}^{\beta,q}_\infty$
are defined by 
\[
\mathcal{N}^{\beta,q}_\infty(\mathbb{R}^N)=\left\{\psi\in L^q(\mathbb{R}^N)\, :\, [\psi]_{\mathcal{N}^{\beta,q}_\infty(\mathbb{R}^N)}<+\infty\right\},\qquad 0<\beta\le 1.
\]
The {\it Sobolev-Slobodecki\u{\i} space} is defined as
\[
W^{\beta,q}(\mathbb{R}^N)=\left\{\psi\in L^q(\mathbb{R}^N)\, :\, [\psi]_{W^{\beta,q}(\mathbb{R}^N)}<+\infty\right\},\qquad 0<\beta<1,
\]
where the seminorm $[\,\cdot\,]_{W^{\beta,q}(\mathbb{R}^N)}$ is given by
\[
[\psi]_{W^{\beta,q}(\mathbb{R}^N)}=\left(\iint_{\mathbb{R}^N\times \mathbb{R}^N} \frac{|\psi(x)-\psi(y)|^q}{|x-y|^{N+\beta\,q}}\,\dd x\,\dd y\right)^\frac{1}{q}.
\]
The above spaces are endowed with the norms
\[
\|\psi\|_{\mathcal{N}^{\beta,q}_\infty(\mathbb{R}^N)}=\|\psi\|_{L^q(\mathbb{R}^N)}+[\psi]_{\mathcal{N}^{\beta,q}_\infty(\mathbb{R}^N)},
\] 
and
\[
\|\psi\|_{W^{\beta,q}(\mathbb{R}^N)}=\|\psi\|_{L^q(\mathbb{R}^N)}+[\psi]_{W^{\beta,q}(\mathbb{R}^N)}.
\]
We also introduce the space $W^{\beta,q}(\Omega)$ for a subset $\Omega\subset \mathbb{R}^N$,
\[
W^{\beta,q}(\Omega)=\left\{\psi\in L^q(\Omega)\, :\, [\psi]_{W^{\beta,q}(\Omega)}<+\infty\right\},\qquad 0<\beta<1,
\]
where naturally
\[
 [\psi]_{W^{\beta,q}(\Omega)}=\left(\iint_{\Omega\times \Omega} \frac{|\psi(x)-\psi(y)|^q}{|x-y|^{N+\beta\,q}}\,\dd x\,\dd y\right)^\frac{1}{q}.
\]
The space $W_{0}^{\beta,q}(\Omega)$ is the subspace of $W^{\beta,q}(\mathbb{R}^N)$ consisting of functions that are
identically zero in the complement of $\Omega$.\\

\textbf{Parabolic Banach spaces:} Suppose that $I\subset\mathbb{R}$ is an interval and let $V$ be a separable and reflexive Banach space,
with a norm $\|\cdot\|_{V}$. Let $V^*$ be its topological dual space. Assume that $v$ is a function such $v(t)$ belongs to $V$ for almost every $t\in I$. If the function $t\mapsto \|v(t)\|_V$ is measurable on $I$ and $1\leq p\leq \infty$, then $v$ belongs to the Banach space $L^p(I;V)$ if and only if
$$
\int_{I}\|v(t)\|_{V}^p\,dt<\infty.
$$
It is well known that the dual space of $L^p(I;V)$ can be characterized by
$$
(L^p(I;V))^*=L^p(I;V^*).
$$
See \cite[Theorem 1.5]{Sh97}. We also write $v\in C(I;V)$ if the function $t\mapsto v(t)$ is continuous with respect to topology on $V$. Moreover, $u$ is locally $\alpha$-H\"older continuous ($0<\alpha<1$) in space (respectively, locally $\beta$-H\"older continuous ($0<\beta<1$) in time) on $\Omega\times I$ and write
$$
u\in C^{\alpha}_{x,\mathrm{loc}}(\Omega\times I) \quad \Big(\text{respectively}, u\in C^{\beta}_{t,\mathrm{loc}}(\Omega\times I)\Big)
$$
if for any compact subset $K\times J\subset \Omega\times I$, we have 
$$
\sup_{t\in J}[u(\cdot,t)]_{C^\alpha(K)}<\infty\quad \Big(\text{respectively}, \sup_{x\in K}[u(x,\cdot)]_{C^\beta(J)}<\infty\Big). 
$$

\textbf{Tail spaces and weak solutions:}  The so-called {\it tail spaces} are expedient for nonlocal equations. The {\it tail space} is defined as
\[
L^{q}_{\alpha}(\mathbb{R}^N)=\left\{u\in L^{q}_{\rm loc}(\mathbb{R}^N)\, :\, \int_{\mathbb{R}^N} \frac{|u|^q}{1+|x|^{N+\alpha}}\,\dd x<+\infty\right\},\qquad q>0 \mbox{ and } \alpha>0,
\]
and is endowed with the norm
$$
\|u\|_{L^{q}_{\alpha}(\mathbb{R}^N)}=\left(\int_{\mathbb{R}^N}\frac{|u(x)|^q}{1+|x|^{N+\alpha}}\, \dd x\right)^\frac{1}{q}.
$$
The  global behavior of a function $u\in L^q_{\alpha}(\mathbb{R}^N)$ is measured by the quantity
\begin{equation*}\label{mtail}
\mathrm{Tail}_{q,\alpha}(u;x_0,R)=\left[R^\alpha\,\int_{\mathbb{R}^N\setminus B_R(x_0)} \frac{|u(x)|^q}{|x-x_0|^{N+\alpha}}\,\dd x\right]^\frac{1}{q}.
\end{equation*}
Here $x_0\in\mathbb{R}^N$, $R>0,\,\beta>0$.

\begin{definition}\label{subsupsolution}(Local weak solution)
Let $1<p\leq 2$ and $0<s<1$. Suppose $\Omega\subset\mathbb{R}^N$ is an open and bounded set and that $f\in L^\infty_{\mathrm{loc}}(\Omega\times I)$. For any $t_0,t_1\in\mathbb{R}$, we define $I=(t_0,t_1]$. We say that $u$ is a local weak solution of the equation
$$
\partial_t u+(-\Delta_p)^s u= f \text{ in }\Omega\times I
$$
if for any closed interval $J=[T_0,T_1]\subset I$, the function $u$ is such that
$$
u\in L^p(J;W^{s,p}_{\mathrm{loc}}(\Omega))\cap L^{p-1}(J;L^{p-1}_{sp}(\mathbb{R}^N))\cap C(J;L^2_{\mathrm{loc}}(\Omega))
$$
and it satisfies
\begin{equation}\label{wksol}
\begin{split}
&-\int_{J}\int_{\Omega}u(x,t)\partial_t \phi(x,t)\,\dd x \dd t+\int_{J}\int_{\mathbb{R}^N}\int_{\mathbb{R}^N}J_p(u(x)-u(y))(\phi(x)-\phi(y))\,\dd \mu\, \dd t\\
&=\int_\Omega u(x,T_0)\phi(x,T_0) \,\dd x-\int_\Omega u(x,T_1)\phi(x,T_1) \,\dd x  \; + \int_{J} \int_{\Omega} f(x,t) \phi(x,t) \dd x \dd t ,
\end{split}
\end{equation}
for every $\phi\in L^p(J;W^{s,p}(\Omega))\cap C^1(J;L^2(\Omega))$ which has spatial support compactly contained in $\Omega$. Here 
$$
\dd \mu=\frac{\dd x \dd y}{|x-y|^{N+ps}}
$$
as defined in \eqref{dmu}.
\end{definition}

We have the following embedding result from \cite[Theorem 2.8]{BLS}.
\begin{theorem}\label{thm:Morrey-type-embedding}
Let $\psi\in \mathcal{N}_\infty^{\beta,q}(\mathbb{R}^N)$, where $0<\beta<1$ and $1\leq q<\infty$ such that $\beta q>N$. Then for every $0<\alpha<\beta-\frac{N}{q}$, we have $\psi\in C^\alpha_{\mathrm{loc}}(\mathbb{R}^N)$. More precisely,
$$
\sup_{x\neq y}\frac{|\psi(x)-\psi(y)|}{|x-y|^\alpha}\leq C\Big([\psi]_{\mathcal{N}_\infty^{\beta,q}(\mathbb{R}^N)}\Big)^\frac{\alpha q+N}{\beta q}\Big(\|\psi\|_{L^q(\mathbb{R}^N)}\Big)^\frac{(\beta-\alpha)q-N}{\beta q},
$$
for some positive constant $C=C(N,q,\alpha,\beta)$ which blows up as $\alpha\nearrow \beta-\frac{N}{q}$.
\end{theorem}

The result below allows for power functions of the solution to be used as test functions. It can be proved by following the proof of \cite[Lemma 3.3]{BLS}}, see also \cite[Lemma 4.1]{GKS23}.

\begin{lemma}\label{tst}
Let $a>0$, $1<p\leq 2$ and $0<s<1$. Suppose that $u$ is a local weak solution of 
$$
\partial_t u+(-\Delta_p)^s u=f \text{ in }B_2\times (-2^{sp}a,0]
$$
as in Definition \ref{subsupsolution},   where 
$$
f \in L_{\mathrm{loc}}^\infty(B_2 \times (-2^{sp}a,0] ) \quad \text{and} \quad u\in L^\infty (E\times (-a,0]),
$$
for every $E\Subset B_2$.
Let $\eta$ be a non-negative Lipschitz function, with compact support in $B_2$. Let $\tau$ be a smooth non-negative function such that $0\leq \tau\leq 1$ and
$$
\tau(t)=0\text{ for }t\leq T_0,\quad \tau(t)=1\text{ for }t\geq T_1
$$
for some $-a<T_0<T_1<0$. Then, for any locally Lipschitz function $F:\mathbb{R}\to \mathbb{R}$ and any $h\in \mathbb{R}^N$ such that $0<|h|<\frac{1}{4}\dist(\mathrm{supp}\,\eta,\partial B_2)$, we have
\begin{equation}\label{tst1}
\begin{split}
&\int_{T_0}^{T_1}\iint_{\mathbb{R}^{2N}}\Big(J_p(u_h(x,t)-u_h(y,t))-J_p(u(x,t)-u(y,t))\Big)\\
&\times\Big(F(\delta_h u(x,t))\eta(x)^2-F(\delta_h u(y,t))\eta(y)^2\Big)\tau(t) \, \dd \mu \dd t\\
&+\int_{B_2}\mathcal{F}(\delta_h u(x,T_1))\eta(x)^2\, \dd x
=\int_{T_0}^{T_1}\int_{B_2}\mathcal{F}(\delta_h u)\eta^2\tau'\,\dd x \dd t   + \int_{T_0}^{T_1} \int_{B_2} (\delta_h f)F(\delta_h u) \eta^2 \tau \dd x \dd t ,
\end{split}
\end{equation}
where $\mathcal{F}(t)=\int_{0}^{t}F(\rho)\,d\rho.$
\end{lemma}
\begin{proof}
The proof follows exactly as in the proof of \cite[Lemma 4.1]{GKS23} as well as \cite[Lemma 3.3]{BLS} by replacing $\eta^p$ with $\eta^2$.

\end{proof}
\section{Improved Besov regularity}
 In this section, we obtain an improvement of regularity on the Besov scale, given some initial H\"older regularity.  
\begin{proposition}\label{prop:Improved-Besov}
    Let $1<p<2$, $0<s<1$ and $a>0$. Suppose that $u$ is a local weak solution of 
\[
\partial_t u + (-\Delta_p)^s u= f \quad \text{in} \quad B_2\times (-2^{sp}a,0]
\] 
 where  
    \[
    \norm{f}_{L^\infty (B_1\times (-a,0]) }\leq 1, \quad \norm{u}_{L^\infty(B_1\times (-a,0])}\leq 1,\]
    
    \[
    \sup_{-a<t\leq 0} \mathrm{Tail}_{p-1,s\,p}(u(\frarg,t);0,1) \leq 1\quad \text{and}  \quad \sup_{-a<t \leq 0}[u(\frarg,t)]_{C^\gamma(B_1)} \leq 1,
    \]
    for some $\gamma \in [0,1)$. Assume also that for some $\alpha \in [0,1)$, $1\leq q< \infty$ and $0<h_0< \frac{1}{10}$, we have
    \begin{equation}\label{h1}
    \int_{T_0}^{0} \sup_{0 < \abs{h} < h_0} \normBB{\frac{\delta_h^2 u}{\abs{h}^\alpha}}_{L^q(B_{R+ h_0})}^q  \dd t< \infty
    \end{equation}
     for $R$ such that $4h_0 < R \leq 1-5h_0$, and for $T_0$ such that $-a < T_0 < 0$. Then for any $\rho$ such that $T_0 <\rho+ T_0<0 $, we have
    \begin{equation}\label{eq:improved-besov}
    \begin{aligned}
        \int_{T_0 +\rho}^{0} &\sup_{0<\abs{h}<h_0} \left\| \frac{\delta_h^2 u}{\abs{h}^{\frac{sp-\gamma(p-2) + \alpha q}{q+1}}}\right\|_{L^{q+1}(B_{R-4h_0})}^{q+1}  \dd t +\sup_{T_0+\rho < t \leq 0} \sup_{0< \abs{h} <h_0} \int_{B_{R-4h_0}} \frac{\abs{\delta_h u(x,t)}^{q + 1}}{\abs{h}^{\alpha q}}  \dd x \\
        &\leq \frac{C}{\min\lbrace\rho, 1 \rbrace} \int_{T_0}^{0}\left( 1+ \sup_{0<\abs{h}<h_0} \left\| \frac{\delta_h^2 u}{\abs{h}^\alpha} \right\|_{L^q(B_{R+h_0})}^q \right) \dd t,
        \end{aligned}
    \end{equation}
    for some positive constant $C=C(N,s,p,q,h_0,\gamma)$.
\end{proposition}
\begin{proof}
We prove the result in two steps. In the first, we prove 
\begin{equation}\label{eq:improved-besov-t1-main}
\begin{aligned}
\int_{T_0+ \rho}^{T_1} \sup_{0< \abs{h}<h_0} \int_{B_{R-4h_0}} & \left| \frac{\delta_h^2 u}{\abs{h}^{\frac{sp-\gamma(p-2) + \alpha q}{q +1}}} \right|^{q + 1} \dd x \dd t + \sup_{0< \abs{h} <h_0} \int_{B_{R-4h_0}} \frac{\abs{\delta_h u(x,T_1)}^{q + 1}}{\abs{h}^{\alpha q}}  \dd x\\
&\leq\frac{C}{\min\lbrace\rho, 1 \rbrace} \int_{T_0}^{T_1}\left( 1+ \sup_{0<\abs{h}<h_0} \left\| \frac{\delta_h^2 u}{\abs{h}^\alpha} \right\|_{L^q(B_{R+h_0})}^q \right) \dd t,	
\end{aligned}
\end{equation}
 with the upper limit $T_1<0$ and in Step 2, we prove \eqref{eq:improved-besov-t1-main} for $T_1=0$, then \eqref{eq:improved-besov} follows by taking the supremum over $T_1$.\\
\textbf{Step 1: Estimate for $T_1<0$. } Assume $\rho+T_0<T_1<0$, let $r= R-4h_0$ and take $\eta$ to be a non-negative smooth function on $\mathbb{R}^N$ such that
    \[
    0\leq\eta\leq 1\text{ on }B_R,\quad 
    \eta \equiv 1 \quad \text{on}\quad B_r,\quad
    \eta\equiv 0\text{ in }\mathbb{R}^N\setminus B_\frac{R+r}{2}\quad \text{and} \quad \abs{\nabla \eta}
   \leq\frac{C}{h_0}\text{ on }B_\frac{R+r}{2},
    \]
     for some positive constant $C$.
    
Take also  $\tau:\R\to \R$ to be a smooth function such that $0 \leq \tau \leq 1$ in $\mathbb{R}$ and 
\[
\tau \equiv 1 \quad \text{on} \quad [T_0+\rho , \infty), \quad \tau \equiv 0 \quad \text{on } (-\infty , T_0], \quad \text{and} \quad \abs{\tau^\prime} \leq \frac{C}{\rho}\text{ in }\mathbb{R},
\]
for some positive constant $C$. \normalcolor Let $1\leq q<\infty$, $0<|h|<h_0$ and $\alpha\in[0,1)$ be as in the statement of the theorem and define $\vartheta=\alpha-\frac{1}{q}$. Lemma \ref{tst}  with $F(t)= J_{q+1}(t)$ implies upon dividing with $\abs{h}^{1+ \vartheta q}$ that
    \begin{equation}\label{eng1}
\begin{split}
\int_{T_0}^{T_1}& \iint_{\mathbb{R}^N\times\mathbb{R}^N} \frac{\Big(J_p(u_h(x,t)-u_h(y,t))-J_p(u(x,t)-u(y,t))\Big)}{|h|^{1+\vartheta q}}\\
&\times\Big(J_{q+1}(\delta_h u(x,t))\,\eta(x)^2-J_{q+1}(\delta_hu(y,t))\,\eta(y)^2\Big)\tau(t)\,\dd \mu\, \dd t\\
&+\frac{1}{q+1}\int_{B_2}\frac{|\delta_h u(x,T_1)|^{q+1}}{|h|^{1+\vartheta q}}\, \eta^2 \dd x\,\\
&=\frac{1}{q+1}\int_{T_0}^{T_1}\int_{B_2}\frac{|\delta_h u|^{q+1}}{|h|^{1+\vartheta q}}\, \eta^2 \,\tau'\, \dd x\, \dd t 
 \; + \int_{T_0}^{T_1} \int_{B_2} \delta_h f \frac{J_{q+1}(\delta_h u)}{\abs{h}^{1+\vartheta q}} \eta^2 \tau \dd x \dd t.
\end{split}
\end{equation}
The triple integral is now divided into three pieces:
\[
\widetilde{\mathcal{I}_i}:=\int_{T_0}^{T_1} {\mathcal{I}_i}(t)\,\tau(t)\,   dt, \qquad i=1,2,3,
\]
where we for $t\in (T_0,T_1)$ define
\[
\begin{split}
\Ical_1(t):=\iint_{B_R\times B_R}& \frac{\Big(J_p(u_h(x,t)-u_h(y,t))-J_p(u(x,t)-u(y,t))\Big)}{|h|^{1+\vartheta q}}\\
&\times\Big(J_{q+1}(\delta_h u(x,t))\,\eta(x)^2-J_{q+1}(\delta_h u(y,t))\,\eta(y)^2\Big)\,\dd\mu ,
\end{split}
\]
\[
\begin{split}
\Ical_2(t):=\displaystyle\iint_{(x,y)\in B_\frac{R+r}{2}\times ( \mathbb{R}^N\setminus B_R)}& \frac{\Big(J_p(u_h(x,t)-u_h(y,t))-J_p(u(x,t)-u(y,t))\Big)}{|h|^{1+\vartheta q}} J_{q+1}(\delta_h u(x,t))\,\eta(x)^2 \dd \mu ,
\end{split}
\]
and
\[
\begin{split}
\Ical_3(t):=-\displaystyle\iint_{(x,y)\in (\mathbb{R}^N\setminus B_R)\times B_\frac{R+r}{2}}& \frac{\Big(J_p(u_h(x,t)-u_h(y,t))-J_p(u(x,t)-u(y,t))\Big)}{|h|^{1+\vartheta q}} J_{q+1}(\delta_h u(y,t))\,\eta(y)^2 \dd \mu,
\end{split}
\]
where we used that $\eta$ vanishes identically outside $B_\frac{R+r}{2}$. We also denote the terms on the right-hand side of \eqref{eng1} by $\Ical_4$ and $\Ical_5$:
$$\Ical_4:= \frac{1}{q+1}\int_{T_0}^{T_1}\int_{B_2}\frac{|\delta_h u|^{q+1}}{|h|^{1+\vartheta q}}\, \eta^2\,\tau'\, \dd x\, \dd t,$$
and
$$  \Ical_5:=  \int_{T_0}^{T_1} \int_{B_2} \delta_h f \frac{J_{q+1}(\delta_h u)}{\abs{h}^{1+\vartheta q}} \eta^2 \tau \dd x \dd t.$$
Therefore, with the above notation, equation \eqref{eng1} reduces to the following equation
\begin{equation}\label{eng2}
\widetilde{\Ical}_1 +\frac{1}{q+1}\int_{B_2}\frac{|\delta_h u(x,T_1)|^{q+1}}{|h|^{1+\vartheta q}}\, \eta^2 \dd x = -\widetilde{\Ical}_2 -\widetilde{\Ical}_3 + \Ical_4.
\end{equation}
\textbf{Estimate of $\widetilde{\mathcal{I}_1}$: Since $1<p<2$}, by the same argument as in the proof of \cite[Pages $10- 13$, Proposition 3.1]{GL}, we obtain for any $t\in(T_0,T_1)$
\begin{equation}\label{estI11}
\begin{split}
 \mathcal{I}_1(t)
\geq c^{-1}
\left[ \frac{\abs{\delta_h u}^{\frac{q-1}{2}}\delta_h u}{\abs{h}^{\frac{1+\vartheta q}{2}}} \eta \right]_{W^{\sigma,2}(B_R)}^2 - C \int_{B_R} \frac{\abs{\delta_h u(x,t)}^{p+q -1}}{\abs{h}^{1+ \vartheta q}} - C \int_{B_R} \frac{\abs{\delta_h u(x,t)}^{q +1}}{\abs{h}^{1+ \vartheta q}}.
\end{split}
\end{equation}
After integration with respect to $t$ this becomes
\begin{equation}\label{estI1}
\begin{split}
\widetilde{\mathcal{I}_1}&
\geq c^{-1}
\int_{T_0}^{T_1}
\left[ \frac{\abs{\delta_h u}^{\frac{q-1}{2}}\delta_h u}{\abs{h}^{\frac{1+\vartheta q}{2}}} \eta \right]_{W^{\sigma,2}(B_R)}^2\tau\,dt - C \int_{T_0}^{T_1}\int_{B_R} \frac{\abs{\delta_h u(x)}^{p+q -1}}{\abs{h}^{1+ \vartheta q}}\tau\dd x\,dt - C \int_{T_0}^{T_1}\int_{B_R} \frac{\abs{\delta_h u(x)}^{q +1}}{\abs{h}^{1+ \vartheta q}} \tau\dd x\,dt,
\end{split}
\end{equation}
where $c=c(p,q)$ and $C=C(N,s,p,q,h_0)$ are positive constants, and $2\sigma = sp -\gamma(p-2)$. We remark that the properties of $\eta$ and the given condition $\sup_{-a\leq t \leq 0}[u(\frarg,t)]_{C^\gamma(B_1)} \leq 1$ are used to derive the above estimate of $\widetilde{\mathcal{I}_1}$. Using \eqref{estI1} in \eqref{eng2}, we arrive at
\begin{equation}\label{eq:sobolev-estimate1}
\begin{aligned}
\int_{T_0}^{T_1} &\left[ \frac{\abs{\delta_h u}^{\frac{q-1}{2}}\delta_h u}{\abs{h}^{\frac{1+\vartheta q}{2}}} \eta \right]_{W^{\sigma,2}(B_R)}^2 \tau \dd t  +\int_{B_2}\frac{|\delta_h u(x,T_1)|^{q+1}}{|h|^{1+\vartheta q}}\, \eta^2 \dd x \\
&\leq C\left( \int_{T_0}^{T_1}\int_{B_R} \frac{\abs{\delta_h u(x,t)}^{q +1}}{\abs{h}^{1+ \vartheta q}} \tau \dd x \dd t + \int_{T_0}^{T_1} \int_{B_R} \frac{\abs{\delta_h u(x,t)}^{p+q -1}}{\abs{h}^{1+ \vartheta q}} \tau \dd x 
 \dd t\right) \\
 &+ c\left(\abs{\widetilde{\Ical}_2}+ \abs{ \widetilde{\Ical}_3} +\abs{\Ical_4} + \abs{\Ical_5}  \right),
 \end{aligned}
\end{equation}
 where $C=C(N,s,p,q,h_0)$ is a positive constant. \\
\textbf{Estimate of the nonlocal terms $\widetilde{\Ical}_2$ and $\widetilde{\Ical}_3$:} First, we estimate $\Ical_2(t)$ for any $t\in (T_0,T_1)$. Since $|u|\leq 1$ in $B_1\times (-a,0]$, we have
\begin{equation}\label{estI2new}
\absb{(J_p(u_h(x,t)-u_h(y,t))-J_p(u(x,t)-u(y,t)))J_{q + 1}(\delta_h u(x,t))} \leq C(p)(1+\abs{u_h(y,t)}^{p-1} + \abs{u(y,t)}^{p-1})\abs{\delta_h u(x,t)}^q,
\end{equation}
where $x \in B_{\frac{R+r}{2}} $, $y \in \R^N \setminus B_{R}$ and $4h_0 < R \leq 1-5h_0$. Therefore, $\abs{x-y} \geq h_0 \abs{y}$ and we get
\begin{equation}\label{estuuh}
\begin{split}
&\int_{\R^N \setminus B_R} \frac{1 + \abs{u(y,t)}^{p-1} + \abs{u_h(y,t)}^{p-1}}{\abs{x-y}^{N+sp}} \dd y \\
&\quad \leq \,C\left(1+\int_{\R^N \setminus B_R} \frac{\abs{u(y,t)}^{p-1}}{\abs{y}^{N+sp}} \dd y
+ \int_{\R^N \setminus B_R} \frac{\abs{u_h(y,t)}^{p-1}}{\abs{y}^{N+sp}} \dd y\right),
\end{split}
\end{equation}
for some positive constant  $C=C(N,s,p,h_0)$.
Now, since $4h_0<R$, $R<1$, $|u|\leq 1$ in $B_1\times(-a,0]$ and $\mathrm{Tail}_{p-1, sp}(u(\cdot,t);0,1)\leq 1$ in $(-a,0]$, we have
\begin{equation}\label{estu}
\begin{split}
\int_{\R^N \setminus B_R} \frac{\abs{u(y,t)}^{p-1}}{\abs{y}^{N+sp}} \dd y &\leq \int_{\R^N \setminus B_1} \frac{\abs{u(y,t)}^{p-1}}{\abs{y}^{N+sp}} \dd y + R^{-N-sp} \int_{B_1} \abs{u(y,t)}^{p-1} \dd y \\
&\leq  \int_{\R^N \setminus B_1} \frac{\abs{u(y,t)}^{p-1}}{\abs{y}^{N+sp}} \dd y + N\omega_N {(4h_0)}^{-N-sp} \\
&\leq C,
\end{split}
\end{equation}
where $C= C(N,s,p,h_0)$. As for $u_h$, we have 
\begin{equation}\label{estuh}
\begin{split}
\int_{\R^N \setminus B_R} \frac{\abs{u(y+h,t)}^{p-1}}{\abs{y}^{N+sp}} \dd y &\leq
\int_{\R^N \setminus B_R(h)} \frac{\abs{u(y,t)}^{p-1}}{\abs{y-h}^{N+sp}} \dd y \leq 
(\frac{3}{2})^{N+sp} \int_{\R^N \setminus B_R(h)} \frac{\abs{u(y,t)}^{p-1}}{\abs{y}^{N+sp}} \dd y \\ 
&\leq (\frac{3}{2})^{N+sp} \left[ \int_{\R^N \setminus B_1} \frac{\abs{u(y,t)}^{p-1}}{\abs{y}^{N+sp}} \dd y + (R-h_0)^{-N-sp} \int_{B_1} \abs{u(y,t)}^{p-1} \dd y \right] \\
&\leq C\left( \mathrm{Tail}_{p-1,s\,p}(u(\cdot,t);0,1)^{p-1} +(3h_0)^{-N-sp} \right) \\
&\leq C,
\end{split}
\end{equation}
for some positive constant $C=C(N,s,p,h_0)$. To obtain the above estimate, we have used that $B_R(h) \subset B_1$, $4h_0<R\leq 1-5h_0$ and 
$$\frac{\abs{y-h}}{\abs{y}} = \abs{\frac{y}{\abs{y}}-\frac{h}{\abs{y}}} \geq \abs{\frac{y}{\abs{y}}} - \abs{\frac{h}{\abs{y}}} \geq 1- {\frac{h_0}{R-h_0}} \geq \frac{2}{3}.$$ Using \eqref{estu} and \eqref{estuh} in \eqref{estuuh}, we get
\begin{equation}\label{estuuh1}
\begin{split}
\int_{\R^N \setminus B_R} \frac{1 + \abs{u(y,t)}^{p-1} + \abs{u_h(y,t)}^{p-1}}{\abs{x-y}^{N+sp}} \dd y &\leq C(N,s,p,h_0).
\end{split}
\end{equation}
Using \eqref{estI2new} and \eqref{estuuh1} and recalling the expression of ${\Ical_2}(t)$, we get for any $t\in (T_0,T_1)$ that
\[
\begin{aligned}
\abs{\Ical_2(t)}&\leq \iint_{(x,y)\in B_{\frac{R+r}{2}}\times {(\R^N \setminus B_R)}} \frac{ \left| J_p(u_h(x,t)-u_h(y,t))-J_p(u(x,t)-u(y,t))\right|}{\abs{x-y}^{N+sp} }\, \frac{\abs{\delta_hu(x,t)}^q}{|h|^{1+\vartheta q}} \eta(x)^2 \,\dd y \dd x \\
& \leq {C(N,s,p,h_0)} \int_{B_\frac{R+r}{2}} \frac{\abs{\delta_h u}^q}{\abs{h}^{1+ \vartheta q}} \eta^2 \dd x.
\end{aligned}
\]
Multiplying by $\tau$ and integrating the above estimate from $T_0$ to $T_1$, we arrive at
\begin{equation}\label{estI2}
\begin{aligned}
\abs{\widetilde{\Ical}_2} &\leq \int_{T_0}^{T_1} \abs{\Ical_2} \tau \dd t 
 \leq C  \left(\int_{T_0}^{T_1} \left(  \int_{B_R} \frac{\abs{\delta_h u}^q}{\abs{h}^{1+ \vartheta q}} \eta^2 \dd x \right) \tau \dd t \right) 
\leq C \int_{T_0}^{T_1} \int_{B_R} \frac{\abs{\delta_h u}^q}{\abs{h}^{1+ \vartheta q}} \dd x \dd t,
\end{aligned}
\end{equation}
for some positive constant $C=C(N,s,p,h_0)$. Here we have used that $0\leq \tau\leq 1$ {on $\mathbb{R}$.
Similarly,
\begin{equation}\label{estI3}
\begin{split}
\abs{\widetilde{\Ical}_3} &\leq C \int_{T_0}^{T_1} \int_{B_R} \frac{\abs{\delta_h u}^q}{\abs{h}^{1+ \vartheta q}} \dd x \dd t,
\end{split}
\end{equation}
for some positive constant $C=C(N,s,p,h_0)$.
}\\
\textbf{Estimate of $\Ical_4$ and $\Ical_5$:} Using the properties of $\eta$ and $\tau$ we have
\begin{equation}\label{estI4}
 \abs{\Ical_4} = \frac{1}{q +1} \left|  \int_{T_0}^{T_1} \int_{B_2} \frac{|\delta_h u|^{q+1}}{|h|^{1+\vartheta q}} \eta^2 \tau^\prime \dd x \dd t \right| \leq \frac{C}{\rho} \int_{T_0}^{T_1} \int_{B_{\frac{R+r}{2}}}\frac{|\delta_h u|^{q+1}}{|h|^{1+\vartheta q}} \dd x \dd t,
\end{equation}
for some positive constant $C=C(N,p,q)$.
By the $L^\infty$-bound on $f$ together with the properties of $\eta$ and $\tau$, we also obtain
\begin{equation}\label{eq:estI5}
\begin{aligned}
    \abs{\Ical_5}=& \left| \int_{T_0}^{T_1}\int_{B_2} \delta_h f \frac{J_{q+1}(\delta_h u)}{\abs{h}^{1+ \vartheta q}} \eta^2 \tau \dd x \dd t  \right| \leq  \int_{T_0}^{T_1}\int_{B_2} \abs{\delta_h f} \frac{\abs{J_{q+1}(\delta_h u)}}{\abs{h}^{1+ \vartheta q}} \eta^2 \tau \dd x \dd t \\
    & \leq 
 \normb{\delta_h f \eta}_{L^\infty(B_2\times (T_0,T_1))}\int_{T_0}^{T_1}\int_{B_2} \frac{\abs{\delta_h u}^{q}}{\abs{h}^{1+ \vartheta q}} \eta \dd x \dd t \\
 & \leq 2\int_{T_0}^{T_1}\int_{B_R} \frac{\abs{\delta_h u}^{q}}{\abs{h}^{1+ \vartheta q}} \dd x \dd t.
    \end{aligned}
\end{equation}
Using \eqref{estI2}, \eqref{estI3}, \eqref{estI4} and \eqref{eq:estI5} in \eqref{eq:sobolev-estimate1}, we arrive at
\begin{equation}\label{eq:sobolev-estimate2}
\begin{aligned}
\int_{T_0 +\rho}^{T_1} &\left[ \frac{\abs{\delta_h u}^{\frac{q-1}{2}}\delta_h u}{\abs{h}^{\frac{1+\vartheta q}{2}}} \eta \right]_{W^{\sigma,2}(B_R)}^2 \dd t  + \frac{c}{q+1}\int_{B_R}\frac{|\delta_h u(x,T_1)|^{q+1}}{|h|^{1+\vartheta q}}\, \eta^2 \dd x \\
&\leq C\left( \int_{T_0}^{T_1}\int_{B_R} \frac{\abs{\delta_h u(x,t)}^{q +1}}{\abs{h}^{1+ \vartheta q}} \tau \dd x \dd t + \int_{T_0}^{T_1} \int_{B_R} \frac{\abs{\delta_h u(x,t)}^{p+q -1}}{\abs{h}^{1+ \vartheta q}} \tau \dd x 
 \dd t\right) \\
 &  + C\left( \frac{1}{\rho} \int_{T_0}^{T_1}\int_{B_R} \frac{\abs{\delta_h u(x,t)}^{q +1}}{\abs{h}^{1+ \vartheta q}} \dd x \dd t + \int_{T_0}^{T_1}\int_{B_R} \frac{\abs{\delta_h u(x,t)}^{q }}{\abs{h}^{1+ \vartheta q}}  \dd x \dd t  \right) \\
 & \leq \frac{C}{\min\{\rho,1\}} \int_{T_0}^{T_1}\int_{B_R} \left(\frac{\abs{\delta_h u(x,t)}^{q}}{\abs{h}^{1+ \vartheta q}} + \frac{\abs{\delta_h u(x,t)}^{q +1}}{\abs{h}^{1+ \vartheta q}} +\frac{\abs{\delta_h u(x,t)}^{p +q - 1}}{\abs{h}^{1+ \vartheta q}}\right) \dd x \dd t \\
 &\leq \frac{C}{{\min\{\rho,1\}}} \int_{T_0}^{T_1}\int_{B_R} \frac{\abs{\delta_h u(x,t)}^{q}}{\abs{h}^{1+ \vartheta q}} \dd x \dd t,
 \end{aligned}
\end{equation}
where $C=C(N,s,p,q,h_0)>0$ is a positive constant. Here we have used that $|\tau|\leq 1$ and that $|\delta_h u(x,t)|\leq 2$ for $(x,t)\in B_R\times (-a,0]\subset B_1\times (-a,0]$. Next, we will find a lower bound of the first integral of the left hand side of \eqref{eq:sobolev-estimate2}, that is
\begin{equation*}
\begin{aligned}
\int_{T_0 +\rho}^{T_1} &\left[ \frac{\abs{\delta_h u}^{\frac{q-1}{2}}\delta_h u}{\abs{h}^{\frac{1+\vartheta q}{2}}} \eta \right]_{W^{\sigma,2}(B_R)}^2 \dd t.
\end{aligned}
\end{equation*}
Indeed, by following the same reasoning as in the proof of \cite[Step $4$, Proposition 3.1]{GL}, we obtain for any $0< \abs{\xi}, \abs{h} < h_0$ that 
\begin{equation}\label{fest1}
\begin{split}
 \int_{B_r} \left| \frac{\delta_\xi \delta_h u}{\abs{\xi}^{\frac{2\sigma}{q+1}} \abs{h}^{\frac{1+\vartheta q}{q +1}} } \right|^{q +1} \dd x &\leq  C \left[  \frac{\abs{\delta_h u}^{\frac{q-1}{2}}\delta_h u}{\abs{h}^{\frac{1+\vartheta q}{2}}} \eta \right]_{W^{\sigma , 2}(B_R)}^2 + C \int_{B_R} \frac{\abs{\delta_h u}^{q +1}}{\abs{h}^{1+ \vartheta q}} \dd x ,
\end{split}
\end{equation}
where $C= C(N,s,q,\sigma,h_0)>0 $ is a positive constant. We take $\xi=h$ in \eqref{fest1}, take supremum over $0< \abs{h} < h_0$ and integrate over the time interval $(T_0+\rho,T_1)$ to arrive at
\begin{equation}\label{eq:secon-order-difference}
\begin{aligned}
\int_{T_0+\rho}^{T_1} \sup_{0<\abs{h}<h_0 }\int_{B_r} \left| \frac{\delta_h^2 u}{\abs{h}^{\frac{2\sigma}{q+1}} \abs{h}^{\frac{1+\vartheta q}{q +1}} } \right|^{q +1}  \dd x \dd t  &\leq C \int_{T_0+ \rho }^{T_1}  \sup_{0<\abs{h}<h_0 }\left[  \frac{\abs{\delta_h u}^{\frac{q-1}{2}}\delta_h u}{\abs{h}^{\frac{1+\vartheta q}{2}}} \eta \right]_{W^{\sigma , 2}(B_R)}^2 \dd t \\
&\quad + C \int_{T_0+ \rho}^{T_1}\sup_{0<\abs{h}<h_0 } \int_{B_R} \frac{\abs{\delta_h u}^{q +1}}{\abs{h}^{1+ \vartheta q}} \dd x \dd t,
\end{aligned}
\end{equation}
where $ C=C(N,s,q,\sigma,h_0)>0 $  is a positive constant. By using \eqref{eq:sobolev-estimate2} in \eqref{eq:secon-order-difference} we obtain

\begin{equation}\label{fest}
\begin{aligned}
\int_{T_0+\rho}^{T_1} & \sup_{0< \abs{h}<h_0 }\int_{B_r} \left| \frac{\delta_h^2 u}{\abs{h}^{\frac{2\sigma}{q+1}+ \frac{1+\vartheta q}{q +1}} } \right|^{q +1} \dd x \dd t + \sup_{0< \abs{h} <h_0} \int_{B_R} \frac{\abs{\delta_h u(x,T_1)}^{q + 1}}{\abs{h}^{1+ \vartheta q}} \eta^2 \dd x \\
 &\leq \frac{C}{\min\{\rho,1\}}  \int_{T_0}^{T_1}  \sup_{0< \abs{h}<h_0} \int_{B_R}  \frac{\abs{\delta_h u(x)}^{q}}{\abs{h}^{1+ \vartheta q}} \dd x \dd t  + C \int_{T_0+ \rho}^{T_1}\sup_{0<\abs{h}<h_0 } \int_{B_R} \frac{\abs{\delta_h u}^{q +1}}{\abs{h}^{1+ \vartheta q}} \dd x \dd t  \\
&  \leq \frac{C}{\min\{\rho,1\}}  \int_{T_0}^{T_1}  \sup_{0< \abs{h}<h_0} \int_{B_R}  \frac{\abs{\delta_h u(x)}^{q}}{\abs{h}^{1+ \vartheta q}} \dd x \dd t ,\normalcolor
\end{aligned}
\end{equation}
where $C$ depends on $N,s,p,q,\sigma, \text{ and }h_0$. In the last inequality, we have again used the fact that $\abs{\delta_h u(x,t)} \leq 2$ for $x \in B_R$ and $t \in (-a,0]$.

Recalling that $\vartheta= \alpha - \frac{1}{q}$ and $2 \sigma = sp-\gamma (p-2)$, we obtain from \eqref{fest} that
\begin{equation}\label{eq:improved-besov-T1new}
\begin{aligned}
\int_{T_0+ \rho}^{T_1} \sup_{0< \abs{h}<h_0} \int_{B_r} & \left| \frac{\delta_h^2 u}{\abs{h}^{\frac{sp-\gamma(p-2) + \alpha q}{q +1}}} \right|^{q + 1} \dd x \dd t + \sup_{0< \abs{h} <h_0} \int_{B_r} \frac{\abs{\delta_h u(x,T_1)}^{q + 1}}{\abs{h}^{\alpha q}}  \dd x\\
&\leq \frac{C}{\min\{\rho,1\}} \int_{T_0}^{T_1} \sup_{0<\abs{h}<h_0} \int_{B_R} \frac{\abs{\delta_h u}^q}{\abs{h}^{\alpha q}} \dd x \dd t,
\end{aligned}
\end{equation}
for some positive constant $C=C(N,s,p,q,h_0,\gamma)$.

Since $1<q<\infty$ and $|u|\leq 1$ in $B_1\times(-a,0]$, we may for $\alpha \in (0,1)$ use the second estimate of \cite[Lemma 2.6]{BrLiSc} to estimate the first order difference quotient in the right hand side of \eqref{eq:improved-besov-T1new} by a second order difference quotient. Then \eqref{eq:improved-besov-T1new} transforms into the
inequality below upon recalling the relations between $R, r$ and $h_0$:
\begin{equation}\label{eq:improved-besov-T1}
\begin{split}
\int_{T_0+ \rho}^{T_1} \sup_{0< \abs{h}<h_0} \int_{B_{R-4h_0}} & \left| \frac{\delta_h^2 u}{\abs{h}^{\frac{sp-\gamma(p-2) + \alpha q}{q +1}}} \right|^{q + 1} \dd x \dd t + \sup_{0< \abs{h} <h_0} \int_{B_{R-4h_0}} \frac{\abs{\delta_h u(x,T_1)}^{q + 1}}{\abs{h}^{\alpha q}}  \dd x\\
&\leq\frac{C}{\min\lbrace\rho, 1 \rbrace} \int_{T_0}^{T_1}\left( 1+ \sup_{0<\abs{h}<h_0} \left\| \frac{\delta_h^2 u}{\abs{h}^\alpha} \right\|_{L^q(B_{R+h_0})}^q \right) \dd t,
\end{split}
\end{equation}
for some positive constant $C=C(N,s,p,q,h_0,\gamma)$. If $\alpha=0$, we may directly use that $|u|\leq 1$ \text{ in } $B_1\times(-a,0]$ together with \eqref{alpha0} to arrive at \eqref{eq:improved-besov-T1}.

\textbf{Step 2: Conclusion for $T_1=0$.}
In this case, the previous proof does not directly work, since it relies on Lemma \ref{tst}, which requires $T_1<0$. However, the constant $C$ in \eqref{eq:improved-besov-T1} does not depend on $T_1$, we can therefore use a limiting argument. By assumption, we have for some $\alpha\in[0,1)$, $1<q<\infty$ and $0<h_0<\frac{1}{10}$, that
$$
\int_{T_0}^{0}\sup_{0<{|h|< h_0}}\left\|\frac{\delta_h ^2 u}{|h|^\alpha}\right\|^q_{L^q(B_{R+h_0})}<\infty
$$
for $4h_0<R\leq 1-5h_0$ and $-a<T_0<0$. For fixed $\rho$ and $T_0$, \eqref{eq:improved-besov-T1} implies that for every $T_1<0$ such that $\rho+T_0<T_1$, there holds
\begin{equation}\label{nlim}
\begin{aligned}
\int_{T_0+ \rho}^{T_1} \sup_{0< \abs{h}<h_0} \int_{B_{R-4h_0}} & \left| \frac{\delta_h^2 u}{\abs{h}^{\frac{sp-\gamma(p-2) + \alpha q}{q +1}}} \right|^{q + 1} \dd x \dd t + \sup_{0< \abs{h} <h_0} \int_{B_{R-4h_0}} \frac{\abs{\delta_h u(x,T_1)}^{q + 1}}{\abs{h}^{\alpha q}}  \dd x\\
&\leq{\frac{C}{\min\lbrace\rho, 1 \rbrace} \int_{T_0}^{T_1}\left( 1+ \sup_{0<\abs{h}<h_0} \left\| \frac{\delta_h^2 u}{\abs{h}^\alpha} \right\|_{L^q(B_{R+h_0})}^q \right) \dd t},
\end{aligned}
\end{equation}
for some positive constant $C=C(N,s,p,q,h_0,\gamma)$.
By the monotone convergence theorem,
\begin{equation}\label{lim1}
\begin{split}
\lim_{T\to 0^{-}} \int_{T_0+ \rho}^{T} \sup_{0< \abs{h}<h_0} \int_{B_{R-4h_0}} \left| \frac{\delta_h^2 u}{\abs{h}^{\frac{sp-\gamma(p-2) + \alpha q}{q +1}}} \right|^{q + 1} \dd x \dd t 
&=\int_{T_0+ \rho}^{0} \sup_{0< \abs{h}<h_0} \int_{B_{R-4h_0}} \left| \frac{\delta_h^2 u}{\abs{h}^{\frac{sp-\gamma(p-2) + \alpha q}{q +1}}} \right|^{q + 1} \dd x \dd t.
\end{split}
\end{equation}
In addition, by the definition of local weak solution
$$
t\mapsto \frac{\delta_h u(\cdot,t)}{|h|^\frac{\alpha q}{q+1}}
$$
is  locally a  continuous function on $(-2^{sp}a,0]$ with values in $L^2(B_r)$ for every fixed $0<|h|<h_0$. Therefore, 
$$
\lim_{T\to 0^{-}}\left\|\frac{\delta_h u(\cdot,T)}{|h|^\frac{\alpha q}{q+1}}-\frac{\delta_h u(\cdot,0)}{|h|^\frac{\alpha q}{q+1}} \right\|_{L^2(B_{R-4h_0})}=0.
$$
Since $q\geq 1$, this in turn implies that{\footnote{We use the following standard fact: if $\{f_n\}_{n\in\mathbb{N}}$ converges to $f$ in $L^m(E)$, then
$$
\lim_{n\to\infty}\inf\|f_n\|_{L^k(E)}\geq \|f\|_{L^k(E)}
$$
for any  $k \geq m$.}}
\begin{equation}\label{lim2}
\lim\inf_{T\to 0^{-}}\Big\|\frac{\delta_h u(\cdot,T)}{|h|^\frac{\alpha q}{q+1}}\Big\|_{L^{q+1}(B_{R-4h_0})}\geq \Big\|\frac{\delta_h u(\cdot,0)}{|h|^\frac{\alpha q}{q+1}}\Big\|_{L^{q+1}(B_{R-4h_0})}
\end{equation}
for every $0<|h|<h_0$. Combining \eqref{lim1} and \eqref{lim2} with \eqref{nlim}, gives  that  estimate \eqref{nlim} holds also for $T_1=0$. Finally, by taking the supremum over $T_1\in [T_0+\rho,0]$ we obtain the desired estimate \eqref{eq:improved-besov}.
\end{proof}

\section{Spatial regularity}
Using Proposition \ref{prop:Improved-Besov}, we establish the following spatial regularity results.

\begin{proposition}\label{prop:improve2}
    Assume that $ a\geq a_0>0$, $ 0<\kappa<1$, $1<p<2$, $0<s<1$, and $\gamma \in [0,1)$. Suppose $u$ is a local weak solution of 
\[
\partial_t u {+} (-\Delta_p)^s u= f \qquad \mbox{ in }{B_2\times (-2^{sp}a,0]}
\] 
{as in Definition \ref{subsupsolution}},  where 
    \[
    \norm{f}_{L^\infty (B_1\times (-a,0]) }\leq 1, \quad 
    \norm{u}_{L^\infty(B_1\times (-a,0])}\leq 1,\]
    \[
    \sup_{-a<t\leq 0} \mathrm{Tail}_{{p-1,s\,p}}(u(\cdot,t);0,1) \leq 1,\quad \text{and} \quad \sup_{-a< t \leq 0}[u(\cdot,t)]_{C^\gamma(B_1)} \leq 1,
    \]
    for some $\gamma \in [0,1)$. Let $\tau = \min \lbrace sp-\gamma(p-2) , 1\rbrace $.  Then  for any $\epsilon \in (0, \tau)$ we have 
    $$\sup_{-a(1-\kappa) \leq t \leq 0}[u(\cdot, t)]_{C^{\tau -\epsilon} (B_{\frac{1}{2}})} \leq C,$$
    for some positive constant $ C=C(N,s,p,\epsilon,\gamma, \kappa a_0)$.
\end{proposition}
\begin{proof}
    Take $0<\epsilon < \tau$ and choose $q_0=q_0(N,\e)$ so that
    $$\tau -\frac{\epsilon}{2} - \frac{N}{q_0} > \tau - \epsilon > 0 \quad \text{and} \quad \frac{q_0}{q_0+1}(\tau-\frac{\epsilon}{4}) \geq \tau -\frac{\epsilon}{2} .$$
    Then consider the sequence
    \[
    \alpha_0 = 0 , \quad \alpha_{i+1}:= \frac{sp-\gamma(p-2) + \alpha_{i} q_0}{q_0 +1}, \quad i= 0, \ldots , i_\infty,
    \]
where we choose $ i_\infty=i_\infty(N,p,s,\e,\gamma)\in\mathbb{N}$ such that 
    $$\alpha_{i_\infty -1} < \tau -\frac{\epsilon}{4}\leq \alpha_{i_\infty} .$$
    This is possible since the sequence of $\alpha_i$ is increasing and 
    $$\lim_{i \to \infty} \alpha_i = sp- \gamma (p-2).$$ 
    Define also
    \[
   h_0= \frac{1}{40 i_\infty},  \quad R_i = \frac{7}{8} - (5i+1)h_0, \quad \text{for} \; i= 0, \ldots , i_\infty .
    \]
    Note that
    \[
    R_0 + h_0 = \frac{7}{8} \quad \text{and} \quad R_{i_\infty -1} - 4h_0 = \frac{3}{4}.
    \]
 Now we apply Proposition \ref{prop:Improved-Besov} (ignoring the second term in the left-hand side of \eqref{eq:improved-besov}) and apply Young's inequality with
    $$q=q_0, \quad \rho = \frac{\kappa a}{2(i_\infty +1)},\quad  T_0 = -a+ (i+1)\rho,\quad R=R_i.$$
 We observe that $0<\rho<-T_0$ for the above choice of $T_0$, since $\kappa\in(0,1)$. Further, $R_i-4h_0 = R_{i+1} +h_0$. Therefore, we obtain the following iterative scheme of inequalities:

    $\bullet$ For $i=0$, we get
    \begin{equation}\label{i0}
    \int_{-a + 2 \rho}^{0} \sup_{0<\abs{h}<h_0} \left\| \frac{\delta_h^2 u}{\abs{h}^{\frac{sp-\gamma(p-2)  }{q_0}}}\right\|_{L^{q_0}(B_{R_1 + h_0})}^{q_0}  \dd t 
    \leq \frac{C}{\min\lbrace\rho, 1 \rbrace} \int_{-a+\rho}^{0} \left( \sup_{0<\abs{h}< h_0}  \norm{\delta_h^2 u}_{L^{q_0}\left(B_{7/8}\right)}^{q_0} +1 \right) \dd t.
    \end{equation}
Notice that $\frac{sp-\gamma(p-2)}{q_0+1} = \alpha_1$.

    $\bullet$ For $i=1, \ldots i_{\infty}-2$, we have
    \begin{equation}\label{i1}
    \int_{-a + (i+2)\rho}^{0} \sup_{0<\abs{h}< h_0} \left\| \frac{\delta_h^2 u}{\abs{h}^{\alpha_{i+1}}}\right\|_{L^{q_{0}}(B_{R_{i+1} + h_0})}^{q_{0}} \dd t 
    \leq \frac{C}{\min\lbrace\rho, 1 \rbrace} \int_{-a + (i+1)\rho}^{0} \left( \sup_{0<\abs{h}< h_0} \left\| \frac{\delta_h^2 u}{\abs{h}^{\alpha_{i}}}\right\|_{L^{q_{0}}(B_{R_{i} + h_0})}^{q_{0}} +1\right) \dd t.
    \end{equation}

    $\bullet$ For $i= i_{\infty}-1$, we have
    \begin{equation}\label{i2}
      \int_{-a + \frac{\kappa a}{2}}^{0} \sup_{0<\abs{h}< h_0} \left\| \frac{\delta_h^2 u}{\abs{h}^{\alpha_{i_{\infty}}}}\right\|_{L^{q_{0}}(B_{3/4})}^{q_{0}} \dd t \leq \frac{C}{\min\lbrace\rho, 1 \rbrace} \int_{-a + i_\infty \rho}^{0} \left( \sup_{0<\abs{h}< h_0} \left\| \frac{\delta_h^2 u}{\abs{h}^{\alpha_{i_{\infty}-1}}}\right\|_{L^{q_{0}}(B_{R_{i_{\infty}-1} + h_0})}^{q_{0}} +1\right)\dd t.
    \end{equation}
    Here $C= C(N,s,p,\epsilon,\gamma,{h_0})$. Also, since $\norm{u}_{L^\infty(B_1{\times(-a,0]})} \leq 1 $, we have
    \begin{equation}\label{usei0}
    \sup_{0< \abs{h}< h_0} \norm{\delta_h^2 u}_{L^{q_0}(B_{7/8})} \leq 3 .
    \end{equation}
    Hence, using \eqref{usei0} in the above iterative scheme of inequalities {\eqref{i0}, \eqref{i1} and \eqref{i2}, we have the following estimate} 
    \begin{equation}\label{eq;:improve-besov-alpha-infty}
         \int_{-a(1- \frac{\kappa}{2}) }^{0} \sup_{0<\abs{h}< h_0} \left\| \frac{\delta_h^2 u}{\abs{h}^{\alpha_{i_{\infty}}}}\right\|_{L^{q_{0}}(B_{3/4})}^{q_{0}} \dd t \leq \frac{C(N,s,p,\epsilon,\gamma,{h_0})}{(\min \{ \kappa a,1\})^{i_\infty}}.
    \end{equation}
    As $\tau -\frac{\epsilon}{4}\leq \alpha_{i_\infty}$, for all $\abs{h} \leq 1$ we have 
    $$\abs{h}^{\alpha_{i_\infty}} \leq \abs{h}^{\tau - \frac{\epsilon}{4}}.$$
    Using this in \eqref{eq;:improve-besov-alpha-infty} we arrive at
    \begin{equation}\label{ib1}
     \int_{-a(1- \frac{\kappa}{2}) }^{0} \sup_{0<\abs{h}< h_0} \left\| \frac{\delta_h^2 u}{\abs{h}^{\tau -\frac{\epsilon}{4}}}\right\|_{L^{q_{0}}(B_{3/4})}^{q_{0}} \dd t \leq \frac{C(N,s,p,\epsilon,\gamma,{h_0})}{(\min \{ \kappa a,1\})^{i_\infty}}.
    \end{equation}
Now we set
    \[
    T_1=0, \quad T_0 = -a(1-\kappa/2), \quad \rho= \frac{\kappa a}{2}, \quad \alpha = \tau - \frac{\epsilon}{4},\quad \text{and}\quad q=q_0.
    \]
    Then,
    \[
    R+h_0=\frac{3}{4}, \quad R- 4h_0 = \frac{3}{4}-5h_0 \geq \frac{5}{8}. 
    \]
    Since $a>0$ and $\kappa\in(0,1)$, we have $0<\rho<T_1-T_0$. We now apply Proposition \ref{prop:Improved-Besov} again, taking \eqref{ib1} into account (ignoring the first term in the left-hand side of \eqref{eq:improved-besov} this time). This yields
    \begin{equation}\label{eq:improved-besov-sup-time}
    \begin{aligned}
        \sup_{-a(1-\kappa)\leq t\leq 0} \, \sup_{0< \abs{h} <h_0} \left\|  \frac{\delta_h u(\frarg,t)}{\abs{h}^\frac{(\tau - \epsilon/4) q_{0}}{q_{0} +1}} \right\|_{L^{q_{0}}(B_{5/8})}^{q_{0}} 
        &\leq \frac{C}{\min\{\kappa a,1\}} \int_{-a(1-\frac{\kappa a}{2})}^{0} \left( \sup_{0<\abs{h}< h_0} \left\| \frac{\delta_h^2 u(\frarg , t)}{\abs{h}^{\tau - \epsilon/4}}\right\|_{L^{q_{{0}}}(B_{3/4})}^{q_{0}} +1 \right) \dd t \\
        &\leq \frac{C}{(\min \{\kappa a,1\})^{i_\infty +1}},
        \end{aligned}
    \end{equation}
for some positive constant $C=C(N,s,p,\epsilon,\gamma,h_0)$. Here we have again used Young's inequality to reduce the power to $q_0$.
    Since $ \frac{q_0}{q_0+1} \geq \frac{\tau -\frac{\epsilon}{2}}{\tau - \frac{\epsilon}{4}}$, we have
    $$\frac{(\tau -\frac{\epsilon}{4}) q_{0}}{ q_{0} +1 } \geq \tau -\frac{\epsilon}{2}.$$
    By using this in \eqref{eq:improved-besov-sup-time} we arrive at
    \begin{equation}\label{eq:improved-besov-tau}
        \begin{aligned}
        \sup_{-a(1-\kappa) \leq t\leq 0} &\, \sup_{0< \abs{h} <h_0} \left\|  \frac{\delta_h u(\frarg,t)}{\abs{h}^{\tau - \frac{\epsilon}{2}}} \right\|_{L^{q_{0}}(B_{5/8})} \\
        &\leq \sup_{-a(1-\kappa)\leq t\leq 0} \, \sup_{0< \abs{h} <h_0} \left\|  \frac{\delta_h u(\frarg,t)}{\abs{h}^\frac{(\tau - \frac{\epsilon}{4}) q_{0}}{q_{0} +1}} \right\|_{L^{q_{0}}(B_{5/8})}  \leq \frac{C(N,\epsilon,p,s,\gamma,h_0)}{(\min\{\kappa a,1\})^{i_\infty +1}}.
        \end{aligned}
    \end{equation}
    Now we take $\chi \in C^\infty_c(B_{9/16})$ such that
    \[
    0\leq \chi \leq 1{\text{ in }B_\frac{9}{16}},\quad \chi {\equiv} 1 \; \text{in} \; B_{1/2}, \quad \abs{\nabla \chi } \leq C{\text{ in }B_\frac{9}{16}}, 
    \]
  for some positive constant $C$. 
    In particular, for any $h$ with $0<\abs{h}< h_0$, we have
    $$ \frac{\abs{\delta_h \chi}}{\abs{h}^{\tau - \frac{\epsilon}{2}}} \leq \frac{\abs{\delta_h \chi}}{\abs{h}} \leq C.$$
    Also, recall that 
    $$\delta_h( \chi u) = \chi_h \delta_h u + u \delta_h \chi .$$
    Hence, using the above properties of $\chi$, 
 for every $t \in [-a(1-\kappa),0],$  we have
    \[
    \begin{aligned}
    &[u\chi]_{\Ncal^{\tau -\epsilon/2,q_{0}}_{\infty}(\R^N)}= \sup_{ \abs{h}>0} \left\|  \frac{\delta_h (u\chi)}{\abs{h}^{\tau - \frac{\epsilon}{2}}} \right\|_{L^{q_{0}}(\R^N)}  \\
    &\leq \sup_{ 0< \abs{h}< h_0} \left\|  \frac{\delta_h (u\chi)}{\abs{h}^{\tau - \frac{\epsilon}{2}}} \right\|_{L^{q_{0}}(\R^N)}  + \sup_{ \abs{h}\geq h_0} \left\|  \frac{\delta_h (u\chi)}{\abs{h}^{\tau - \frac{\epsilon}{2}}} \right\|_{L^{q_{0}}(\R^N)}   \\
   &  \leq  \sup_{ 0< \abs{h}< h_0} \left\|  \frac{\chi_h \delta_h u}{\abs{h}^{\tau - \frac{\epsilon}{2}}} \right\|_{L^{q_{0}}(\R^N)}  +   \sup_{ 0< \abs{h}< h_0} \left\|  \frac{u \delta_h \chi}{\abs{h}^{\tau - \frac{\epsilon}{2}}} \right\|_{L^{q_{0}}(\R^N)} \\
   &\quad \qquad + \frac{1}{h_0^{\tau - \frac{\epsilon}{2}}} \sup_{\abs{h} \geq h_0} \norm{\delta_h(u \chi)}_{L^{q_{0}}(\R^N)} \\
   &\leq \sup_{ 0< \abs{h}< h_0} \left\|  \frac{ \delta_h u}{\abs{h}^{\tau - \frac{\epsilon}{2}}} \right\|_{L^{q_{0}}(B_{\frac{9}{16} + h_0})}  + \norm{u}_{L^{q_{0}}(B_{\frac{9}{16} + h_0})} \sup_{0 < \abs{h} < h_0} \left \| \frac{\delta_h \chi}{\abs{h}^{\tau - \frac{\epsilon}{2}}}  \right\|_{L^\infty(B_{\frac{9}{16} +h_0})} \\
   & \qquad + \frac{1}{h_0^{\tau - \frac{\epsilon}{2}}} \sup_{\abs{h} \geq h_0} \left( \norm{(u \chi)_h}_{L^{q_{0}}(B_{\frac{9}{16}}(-h))} + \norm{u \chi}_{L^{q_{0}}(B_{\frac{9}{16}})} \right) \\
   & \leq \sup_{ 0< \abs{h}< h_0} \left\|  \frac{ \delta_h u}{\abs{h}^{\tau - \frac{\epsilon}{2}}} \right\|_{L^{q_{0}}(B_{\frac{5}{8}})}  + C\norm{u}_{L^{q_{0}}(B_1)},
    \end{aligned}
    \]
    {for some positive constant $C=C(N,s,p,\epsilon,\gamma,h_0)$.} Since $h_0$ depends on $s,p,\epsilon,\gamma$ and $\norm{u}_{L^\infty(B_1\times(-a,0])} \leq 1 $, the above estimate combined with \eqref{eq:improved-besov-tau} implies 
    \begin{equation}\label{nl}
        \sup_{a(1-\kappa) \leq t \leq 0} [u\chi(\cdot,t)]_{N^{\tau -\epsilon/2,q_{0}}_{\infty}(\R^N)} \leq \frac{C(N,s,p,\epsilon,\gamma)}{(\min\{\kappa a,1\})^{c(N,s,p,\e,\gamma)}}.
    \end{equation}
    Notice that by the choice of $q_0$, we have
    \[
    \tau - \epsilon < \tau - \frac{\epsilon}{2} - \frac{N}{q_0}. 
    \]    
    Now we apply Theorem \ref{thm:Morrey-type-embedding} with $q=q_{0}$, $\beta= \tau - \epsilon/2$ and $\alpha = \tau - \epsilon$ to obtain 
    \[
    \begin{aligned}\
    &\sup_{-a(1-\kappa)\leq t \leq 0} [u(\cdot,t)]_{C^{\tau - \epsilon}(B_{1/2})} = \sup_{-a(1-\kappa) \leq t \leq 0} [(u \chi)(\frarg ,t))]_{C^{\tau - \epsilon}(B_{1/2})} \leq \sup_{-a(1-\kappa) \leq t \leq 0} [(u\chi)(\cdot ,t)]_{C^{\tau - \epsilon}(\R^N)} \\
    &\leq C \sup_{-a(1-\kappa) \leq t \leq 0} \left( \left( [u\chi(\cdot,t)]_{\Ncal^{\tau - \frac{\epsilon}{2}, q_0}_\infty(\R^N)}\right)^{\frac{(\tau-\epsilon)q_{0}+ N}{(\tau - \epsilon/2)(q_{0})}} \left(\norm{u \chi(\cdot,t)}_{L^{q_{0}}(\R^N)}\right)^{\frac{\frac{(q_{0})\epsilon}{2}-N}{(\tau -\epsilon/2)q_{0}}} \right) \\
    &\leq C \left( \sup_{-a(1-\kappa) \leq t \leq 0} [u\chi(\cdot,t)]_{\Ncal^{\tau - \frac{\epsilon}{2}, q_{0}}_\infty(\R^N)}\right)^{\frac{(\tau-\epsilon)q_{0}+ N}{(\tau - \epsilon/2)(q_{0})}}
    \left(\sup_{-a(1-\kappa) \leq t \leq 0}\norm{u \chi(\cdot,t)}_{L^{q_{0}}(\R^N)}\right)^{\frac{\frac{\epsilon q_0}{2}-N}{(\tau -\epsilon/2)q_{0}}}\\
    & \leq \frac{C}{(\min\{\kappa a,1\})^{c\frac{(\tau-\epsilon)q_{0}+ N}{(\tau - \epsilon/2)q_{0}}}}, 
    \end{aligned}
    \]
     where we to obtain the above estimate, have also used the estimate \eqref{nl} and the fact that $\|u\|_{L^\infty(B_1\times (-a,0])}\leq 1$.  Here $C=C(N,s,p,\epsilon,\gamma)$ and $c=c(N,s,p,\e,\gamma)$ are positive constants. The proof follows by using that $\min\lbrace\kappa a,1\rbrace\geq \min \lbrace \kappa a_0,1\rbrace$ since $a\geq a_0$.
\end{proof}

By yet another iteration process, we arrive at the final almost $sp/(p-1)$-regularity in space.

\begin{theorem}[Almost $sp/(p-1)$-regularity]
\label{teo:localalmost}
Let $1<p<2$ and $0<s<1$. Suppose $u$ is a local weak solution of 
\[
\partial_t u+(-\Delta_p)^s u= f \qquad \mbox{ in }B_2\times (-2^{sp},0]
\] 
 where 
\[
\begin{aligned}
&\|u\|_{L^\infty(B_1\times (-1,0])}\leq 1, \qquad \norm{f}_{L^\infty(B_1 \times (-1,0]))} \leq 1 , \\
&\mbox{ and }\qquad \sup_{t\in (-1,0]}\mathrm{Tail}_{p-1,s\,p}(u(\cdot,t);0,1)^{p-1}\leq 1.
\end{aligned}
\]
Then for any $\e\in(0,\Theta)$, there is { $\sigma(\e,s,p)\in(0,\frac{1}{2}]$,} such that $u(\cdot,t)\in C^{{\Theta}-\e}(B_\sigma)$ for all $t\in (-\sigma^{sp},0]$, where
$$
{\Theta} = \min(sp/(p-1),1).
$$
Moreover, 
$$
\sup_{t\in (-\sigma^{sp},0]}[u(\cdot,t)]_{C^{{\Theta}-\e}(B_\sigma)}\leq C(s,p,\e,N).
$$
\end{theorem}
\par
\begin{proof} The idea is to apply Proposition \ref{prop:improve2} iteratively. Take $\e\in(0,\Theta)$ and define
$$
\gamma_0=0,\qquad \gamma_{i+1}=sp-\gamma_i(p-2)-\frac{\e(p-1)}{2}.
$$
Then $\{\gamma_i\}$ is an increasing sequence and $\gamma_i\to sp/(p-1)-\e/2$, as $i\to\infty$.  {It is clear that there is {$i_\infty=i_\infty(s,p,\e)\in\mathbb{N}$} such that $\gamma_{i_\infty}\geq \Theta-\e$ and $\gamma_{i_\infty-1}<1$.}

Define 
$$
v_i(x,t)=\frac{u(2^{-i} x,2^{-isp}{M_i}^{2-p}\,t)}{M_i}
$$
and
\begin{equation}\label{eq:Mi}
\begin{split}
M_i &= 1 + \|u\|_{L^\infty(B_{2^{-i}}\times ({-2^{-i sp}},0])}+ \sup_{t\in (-2^{-isp},0]}\mathrm{Tail}_{p-1,s\,p}(u(\cdot,t);0,2^{-i}) \\
& \qquad +  2^{-isp/(p-1)} \norm{f}_{L^\infty(B_{2^{-i}}\times( -2^{-isp} ,0])}^{\frac{1}{p-1}} + 2^{-i\gamma_i}\sup_{t\in (2^{-i sp},0]}[u(\cdot,t)]_{C^{\gamma_i}(B_{2^{-i}})}\\
&\leq C(N,p,s,\e)\Big(1+\sup_{t\in (-2^{-i sp},0]}[u(\cdot,t)]_{C^{\gamma_i}(B_{2^{-i}})})\Big)\\
&=C(N,p,s,\e)\Big(1+M_{i-1}2^{(i-1)\gamma_i}\sup_{t\in (-M_{i-1}^{p-2}2^{-i sp},0]}[v_{i-1}(\cdot,t)]_{C^{\gamma_i}(B_{\frac12})}\Big)\\
&\leq C(N,p,s,\e)\Big(1+M_{i-1}\sup_{t\in (-M_{i-1}^{p-2}2^{-i sp},0]}[v_{i-1}(\cdot,t)]_{C^{\gamma_i}(B_{\frac12})}\Big).
\end{split}
\end{equation}
Then $v_i$ is a {local weak solution of the equation $v_t+(-\Delta_p)^sv= f_i$} in $B_2\times(-2^{sp} M_i^{p-2},0]$, where

$$f_i(x,t)=2^{-isp} \frac{f(2^{-i}x,2^{-isp}M_i^{2-p}t)}{M_i^{p-1}}.$$
Moreover, if  
$$\sup_{t\in (-2^{-isp},0]}[u(\cdot,t)]_{C^{\gamma_i}(B_{2^{-i}})}<\infty,
$$
then $M_i <\infty$ and $v_i$ and $f_i$ are well defined and satisfy
$$
\|v_i\|_{L^\infty(B_1\times (-M_i^{p-2},0])}\leq 1,\qquad  \sup_{t\in (-M_i^{p-2},0]}\mathrm{Tail}_{p-1,s\,p}(v_i(\cdot,t);0,1)^{p-1}\leq 1,$$
$$ \sup_{t\in (-M_i^{p-2},0]}[v_i(\cdot,t)]_{C^{\gamma_i}(B_{1})}\leq 1,
 \quad \text{and} \quad 
\| f_i\|_{L^\infty(B_1\times (-M_i^{p-2},0])}\leq 1.
$$
Now we apply Proposition \ref{prop:improve2} with  $a=M_i^{p-2}$ and $\kappa = 1-2^{-sp}$ to $v_i$ in the cylinders $B_2\times (-2^{sp}M_i^{p-2},0]$ successively with $\gamma=\gamma_i$ and $\e$ replaced by $\frac{\e(p-1)}{2}$ and obtain for $i=1$ to $i=i_\infty$
\begin{eqnarray}
\sup_{t\in (-2^{-sp},0]}[v_0(\cdot,t)]_{C^{\gamma_1}(B_\frac12)}\leq C(s,p,\e,N), \nonumber \\
 M_1\leq C(s,p,\e,N)\Big(1+ \sup_{t\in (-2^{-sp},0]}[v_0(\cdot,t)]_{C^{\gamma_1}(B_\frac12)}\Big)\leq C(s,p,\e,N),\nonumber \\
 \sup_{t\in (-M_1^{p-2}2^{-sp},0]}[v_1(\cdot,t)]_{C^{\gamma_2}(B_\frac12)}\leq C(s,p,\e,N),\nonumber \\
 M_2\leq C(s,p,\e,N)\Big(1+M_1 \sup_{t\in (-M_1^{p-2}2^{-sp},0]}[v_1(\cdot,t)]_{C^{\gamma_2}(B_\frac12)}\Big)\leq C(s,p,\e,N),\nonumber \\
 \sup_{t\in (-M_{i-1}^{p-2}2^{-sp},0]}[v_{i-1}(\cdot,t)]_{C^{\gamma_{i}}(B_\frac12)}\leq C(s,p,\e,N),\nonumber \\
 M_{i}\leq C(N,p,s,\e)\Big(1+M_{i-1}\sup_{t\in (-M_{i-1}^{p-2}2^{-i sp},0]}[v_{i-1}(\cdot,t)]_{C^{\gamma_i}(B_{\frac12})}\Big)\leq C(s,p,\e,N),\nonumber \\
\cdots\nonumber  \\
 \sup_{t\in (-M_{i_\infty-1}^{p-2}2^{-sp},0]} [v_{i_{\infty}-1}(\cdot,t)]_{C^{\min\big(\gamma_{i_{\infty}},1-\frac{\e(p-1)}{2}\big)}(B_\frac12)}\leq C(s,p,\e,N). \label{eq:lastest}
\end{eqnarray}
Note that at each step above we have used Proposition \ref{prop:improve2} for a time interval with $a=M_i^{p-2}\geq C(s,p,\e,N)=a_0$ and we have also used \eqref{eq:Mi} to estimate $M_i$. If $sp>1$, we only do one iteration and in particular $i_\infty=1$. In particular in that case \eqref{eq:lastest} becomes
\[
\sup_{t \in (-2^{-sp},0]} [v_0(\frarg,t)]_{C^{sp - \frac{\epsilon(p-1)}{2}}}\leq C(s,p,\epsilon,N).
\]
Scaling back to $u$ and using that $\gamma_{i_\infty}\geq \Theta-\e$, we obtain 
\[
\begin{split}
\sup_{t\in (-2^{-(i_\infty)sp},0]}[u(\cdot,t)]_{C^{\Theta-\e}(B_{2^{-i_\infty}})}&=\sup_{t\in (-M_{i_\infty-1}^{p-2}2^{-sp},0]}2^{(i_\infty-1)(\Theta-\e)}M_{i_\infty-1}[v_{i_{\infty-1}}(\cdot,t)]_{C^{\Theta-\e}(B_\frac12)}\\
&\leq C(s,p,\e,N).
\end{split}
\]
This is the desired result with $\sigma = 2^{-i_\infty}$.
\end{proof}

{Now we apply Theorem \ref {teo:localalmost} to prove the following spatial regularity result.}

\begin{theorem}[Spatial almost $C^{sp/(p-1)}$ regularity]
\label{teo:new}
Let $\Omega\subset\mathbb{R}^N$ be a bounded and open set, $I=(t_0,t_1]$, {$1<p<2$} and $0<s<1$. Suppose $u$ is a local weak solution of 
\[
\partial_t u+(-\Delta_p)^s u= f \qquad \mbox{ in }\Omega\times I,
\]
with $f \in L^\infty_\loc(\Omega \times I)$ such that 
$$
\norm{u}_{L^\infty(Q_{2R,(2R)^{s\,p}}(x_0,T_0))}+\sup_{-(2R)^{sp}+T_0<t\leq T_0} \mathrm{Tail}_{p-1,s\,p}(u(\frarg,t);x_0,R)   <\infty,
$$
where $(x_0,T_0)$ are such that 
\[
Q_{2R,(2R)^{s\,p}}(x_0,T_0)\Subset\Omega\times I.
\] 
Then $u\in C^{\Theta-\e}_{x,\rm loc}(\Omega\times I)$ for every {$\e\in(0,\Theta)$}, where $\Theta = \min(sp/(p-1),1)$. 
\par
More precisely, for every $\e\in(0,\Theta)$, $R>0$ and every $(x_0,T_0)$ such that 
\[
Q_{2R,(2R)^{s\,p}}(x_0,T_0)\Subset\Omega\times I,
\] 
there  exist constants  $C=C(N,s,p,\e)>0$ and $ 0<\sigma(\epsilon,s,p)\leq \frac{1}{2}$ such that
\begin{equation}
\label{eq:sigmaest}
\sup_{t\in \left(T_0-(\sigma R)^{s\,p},T_0\right]} [u(\cdot,t)]_{C^{\Theta-\e}(B_{{\frac{\sigma R}{2}}}(x_0))}\leq
{C\Mcal^{1+\frac{2-p}{sp}}R^{-(\Theta-\e)},}
\end{equation}
where
\[
\Mcal= \|u\|_{L^\infty(Q_{R,R^{sp}}(x_0,T_0))}+ \sup_{t\in (T_0-R^{sp},T_0]}\mathrm{Tail}_{p-1,s\,p}(u(\cdot,t);x_0,R)^{p-1} + R^{sp} \norm{f}_{L^\infty(Q_{R,R^{sp}}(x_0,T_0))} +1.
\]
\end{theorem}
\begin{proof}
We perform the proof in the case $x_0=0$ and $T_0=0$.\\
Let
\[
u_R(x,t):=\frac{1}{M}\,u(RM^\frac{p-2}{sp} x+y_0,R^{sp}t),\qquad \mbox{ for }x\in B_2,\, t\in(-2^{sp},0] 
\]
where $M=M(R)$ is given by
\[
\begin{aligned}
M=&\left(2+{\left(\frac{2 N \omega(N)}{sp}\right)^\frac{1}{p-1} }\right) \|u\|_{L^\infty(B_{R}\times (-R^{sp},0])}+ R^{sp} \norm{f}_{L^{\infty}(B_R \times(-R^{sp},0])}
\\
& +2(1-\sigma)^{-N-sp}\sup_{t\in (-R^{sp},0]}\mathrm{Tail}_{p-1,s\,p}(u(\cdot,t);0,R)^{p-1}+1
\end{aligned}
\]
and $y_0$ is chosen so that 
\begin{equation}
\label{x0ass}
\sigma RM^\frac{p-2}{sp}+|y_0|\leq \sigma R
\end{equation}
where  $\sigma\in(0,\frac{1}{2}]$  is as in Theorem \ref{teo:localalmost}. Since $1<p<2$ and $M\geq 1$ 
, \eqref{x0ass} implies
$$
2RM^\frac{p-2}{sp}+|y_0| \leq 2RM^\frac{p-2}{sp} + \sigma R - \sigma R M^{\frac{p-2}{sp}}=RM^\frac{p-2}{sp}(2-\sigma)+\sigma R \leq 2R.
$$
Therefore $u_R$ is a local weak solution of $\partial_t u+(-\Delta_p)^s u= \tilde{f}$ in $B_2\times (-2^{sp},0]$, where
$$\tilde{f}(x,t)=\frac{R^{sp}}{M} f(RM^{\frac{p-2}{sp}}x +y_0,R^{sp}t) .$$

It is also straight forward to verify that \eqref{x0ass} implies
$$
\|u_R\|_{L^\infty(B_1\times (-1,0])}=M^{-1}\|u\|_{L^\infty(B_{RM^{\tiny\frac{p-2}{sp}}}(y_0)\times (-R^{sp},0])}\leq M^{-1}\|u\|_{L^\infty(B_{R}\times (-R^{sp},0])}\leq 1,
$$
as well as
$$\norm{\tilde{f}}_{B_1\times(-1,0]} = M^{-1}R^{sp} \norm{f}_{L^\infty(B_{RM^{\tiny\frac{p-2}{sp}}}(y_0)\times (-R^{sp},0])} \leq M^{-1}R^{sp} \norm{f}_{L^\infty(B_R\times(-R^{sp},0])} \leq 1.$$

We will now verify that also 
$$
\sup_{t\in (-1,0]}\int_{\mathbb{R}^N\setminus B_1}\frac{|u_R(z,t)|^{p-1}}{|{z}|^{N+s\,p}}\,  \dd z\leq 1.
$$
We have
\[
\begin{split}
&\int_{\mathbb{R}^N\setminus B_1}\frac{|u_R(z,t)|^{p-1}}{|z|^{N+s\,p}}\,  dz=R^{sp}M^{(p-2)-(p-1)}\int_{\R^N\setminus B_{RM^\frac{p-2}{sp}}(y_0)}\frac{|u(y,R^{sp}t)|^{p-1}}{|y-y_0|^{N+s\,p}}\,  \dd y\\
&=R^{sp}M^{-1}\left(\int_{B_R \setminus B_{RM^{\tiny \frac{p-2}{sp}}}(y_0)}\frac{|u(y,R^{sp}t)|^{p-1}}{|y-y_0|^{N+s\,p}}\,  \dd y  +\int_{\R^N\setminus B_R}\frac{|u(y,R^{sp}t)|^{p-1}}{|y-y_0|^{N+s\,p}}\,  \dd y  \right)\\
&=I_1+I_2.
\end{split}
\]

We note that $|y_0|\leq \sigma R$, which implies that if $y\in B_R^c$, then $|y-y_0|\geq (1-\sigma) {|y|}$. Therefore 
$$
\sup_{t\in (-1,0]}I_2\leq R^{sp}M^{-1}(1-\sigma)^{-N-sp}\sup_{t\in (-R^{sp},0]}\int_{\R^N\setminus B_R} \frac{|u(y,t)|^{p-1}}{|y|^{N+s\,p}}\,  \dd y\leq \frac 12
$$ 
by the choice of $M$. For $I_1$, we instead have
\[
\begin{split}
\sup_{t\in (-1,0]}I_1&\leq \|u\|^{p-1}_{L^\infty(B_{R}\times (-R^{sp},0])}R^{sp}M^{-1}\int_{\R^N \setminus B_{RM^\frac{p-2}{sp}}}\frac{1}{|\tilde y|^{N+s\,p}}\,  d\tilde y\\
&=\|u\|^{p-1}_{L^\infty(B_{R}\times (-R^{sp},0])}R^{sp}M^{-1}\int^\infty_{RM^\frac{p-2}{sp}}{ N\omega(N)} r^{-1-sp}   dr\\
&=\|u\|^{p-1}_{L^\infty(B_{R}\times (-R^{sp},0])}\frac{M^{-(p-1)}{ N\omega(N)} }{sp}\\
&\leq \frac 12,
\end{split}
\]
{again by the choice of $M$}.
By Theorem \ref{teo:localalmost}, $u_R$ satisfies the estimate
\[
\sup_{t\in (-\sigma^{sp},0]}[u_R(x,t)]_{C^{\Theta-\e}(B_{\sigma })}\leq C,\quad C=C(N,s,p,\e).
\]
By scaling back, we obtain
\begin{equation}
\label{notfinalest}
\sup_{t\in (-(\sigma R)^{sp},0]}[u(x,t)]_{C^{\Theta-\e}\left(B_{\sigma RM^{\frac{p-2}{sp}}(y_0) }\right)}\leq CM^{1+\frac{2-p}{sp} (\Theta-\e)}R^{-(\Theta-\e)}.
\end{equation}
Since this is valid for all $y_0$ satisfying \eqref{x0ass}, we can, by varying $y_0$, cover the whole $B_{\sigma R/2}$ and obtain
\begin{equation}
\label{eq:finalest}
\sup_{t\in (-(\sigma R)^{sp},0]}[u(x,t)]_{C^{\Theta-\e}(B_{\frac{\sigma R}{2}})}\leq CM^{1+\frac{2-p}{sp}}R^{-(\Theta-\e)}.
\end{equation}
Once this is settled, this proves the theorem.

Let us provide the details of how to vary $y_0$. First of all, if $M^\frac{p-2}{sp}\geq 1/2$, then \eqref{notfinalest} with $y_0=0$ implies \eqref{eq:finalest} directly. Therefore, we assume $M^\frac{p-2}{sp}{<}1/2$ in what follows.
We need to estimate 
$$
\frac{|u(x,t)-u(y,t)|}{|x-y|^{\Theta-\e}},\quad x,y\in B_\frac{\sigma R}{2},t\in (-(\sigma R)^{sp},0].
$$
Take $x,y\in B_{\sigma R/2}$ and fix $t\in (-(\sigma R)^{sp},0]$. If $|x-y|\leq  \sigma RM^\frac{p-2}{sp}$, it follows that $$x,y\in B_{\sigma RM^\frac{p-2}{sp}}((x-y)/2)$$ and 
$$
\sigma RM^\frac{p-2}{sp}+|(x-y)/2|\leq \sigma R/2+\sigma R/4\leq  3\sigma R/4\leq \sigma R$$
so that the choice $y_0=(x-y)/2$ is admissible for \eqref{x0ass}. We may therefore apply \eqref{notfinalest} directly by choosing ${y_0}=(x-y)/2$ to obtain 
$$
\frac{|u(x,t)-u(y,t)|}{|x-y|^{\Theta-\e}}\leq CM^{1+\frac{2-p}{sp} (\Theta-\e)}R^{-(\Theta-\e)} { \leq C M^{1+ \frac{2-p}{sp}} R^{-(\Theta-\e)}}.
$$
 If instead $|x-y|\geq \sigma  RM^\frac{p-2}{sp}$, we take
 $$z_i=x+\frac{y-x}{\lceil M^\frac{2-p}{sp}\rceil}i, \quad i=0, \ldots,  J=\lceil M^\frac{2-p}{{sp}}\rceil.
$$
Since $x,y\in B_{\sigma R/2}$, we have $|z_i-z_{i+1}| = \frac{\abs{x-y}}{\lceil M^\frac{2-p}{sp}\rceil} < \sigma RM^\frac{p-2}{sp}$, $z_0=x$ and $z_J=y$. This choice implies that $$z_i,z_{i+1}\in B_{{\sigma} RM^\frac{p-2}{sp}/2}((z_i-z_{i+1})/2)$$ and
$$
\sigma RM^\frac{p-2}{sp}+|(z_i-z_{i+1})/2|\leq \sigma R/2+\sigma R/4 \leq \sigma R,
$$
as before. Therefore, the choice $y_0=(z_i-z_{i+1})/2$ is admissible for \eqref{x0ass}. We can therefore apply \eqref{notfinalest}  with $y_0=(z_i-z_{i+1})/2$ and obtain 
$$
\frac{|u(z_i,t)-u(z_{i+1},t)|}{|z_i-z_{i+1}|^{\Theta-\e}}\leq CM^{1+\frac{2-p}{sp} (\Theta-\e)}R^{-(\Theta-\e)}.
$$
By the triangle inequality, we can now conclude
\[
\begin{split}
|u(x,t)-u(y,t)|\leq \sum_{i=0}^J|u(z_i,t)-u(z_{i+1},t)|&\leq \sum_{i=0}^JCM^{1+\frac{2-p}{sp} (\Theta-\e)}R^{-(\Theta-\e)}|z_i-z_{i+1}|^{\Theta-\e}\\
&{\leq \sum_{i=0}^JCM^{1+\frac{2-p}{sp} (\Theta-\e)}R^{-(\Theta-\e)} \left(\frac{\abs{x-y}}{\lceil M^{\frac{2-p}{sp}}\rceil}\right)^{\Theta-\e}}\\
&\leq C J M^{1+\frac{2-p}{sp} (\Theta-\e)}R^{-(\Theta -\epsilon)} M^{\frac{p-2}{sp}(\Theta -\e)} \abs{x-y}^{\Theta -\e} \\
&\leq CM^{1+\frac{2-p}{sp}}R^{-(\Theta-\e)}|x-y|^{\Theta -\e} 
\end{split}
\]
where we have used that $J= \lceil M^{\frac{2-p}{sp}}\rceil \leq M^\frac{2-p}{sp}+1\leq 2M^\frac{2-p}{sp}$.
This proves \eqref{eq:finalest} and the proof is complete.
\end{proof}
\begin{remark}\label{rem:covering} We note that by a covering argument, we may, under the assumptions of Theorem \ref{teo:new}, obtain the estimate
\begin{equation}
\label{eq:nonsigmaest}
\sup_{t\in \left(T_0-(\frac{R}{2})^{s\,p},T_0\right]} [u(\cdot,t)]_{C^{\Theta-\e}(B_{ \frac{R}{2}}(x_0))}\leq C\Mcal^{1+\frac{2-p}{sp}}R^{-(\Theta-\e)},
\end{equation}
where
$$
\Mcal= \|u\|_{L^\infty(Q_{R,R^{sp}}(x_0,T_0))}+ \sup_{t\in (T_0-R^{sp},T_0]}\mathrm{Tail}_{p-1,s\,p}(u(\cdot,t);x_0,R)^{p-1} + R^{sp} \norm{f}_{L^\infty(Q_{R,R^{sp}}(x_0,T_0))} +1.
$$
Indeed, we may cover $B_{\frac{R}{2}(x_0)}\times (-(\frac{R}{2})^{{sp}}+T_0,T_0]$ with a finite number of cylinders of the form 
$$
B_{\frac{\sigma R}{2}}(x_i)\times \Big(-(\frac{\sigma R}{2})^{sp}+t_j,t_j\Big],\quad i=1,\ldots,n,\quad j=1,\ldots,m,
$$
where $B_{\frac{R}{2}}(x_i)\subset B_R(x_0)$ and $(-(\frac{R}{2})^{sp}+t_j,t_j]\subset (-R^{sp}+T_0,T_0]$. We may then apply Theorem \ref{teo:new} to each of these cylinders and obtain 
$$
{\sup_{t\in \left[t_j-(\frac{\sigma R}{2})^{s\,p},t_j\right]} [u(\cdot,t)]_{C^{\Theta-\e}(B_{\frac{\sigma R}{2}}(x_i))}\leq CM^{1+\frac{2-p}{sp}}R^{-(\Theta-\e)},}
$$
where
$$
\Mcal= \|u\|_{L^\infty(Q_{R,R^{sp}}(x_0,T_0))}+ \sup_{t\in (T_0-R^{sp},T_0]}\mathrm{Tail}_{p-1,s\,p}(u(\cdot,t);x_0,R)^{p-1} + R^{sp} \norm{f}_{L^\infty(Q_{R,R^{sp}}(x_0,T_0))} +1.
$$
By the triangle inequality, this implies the desired inequality \eqref{eq:nonsigmaest}.
\end{remark}

\section{Time regularity}
 In this section, we prove H\"older continuity in the time variable, assuming H\"older regularity in the spatial variable. For $u\in L^1(Q_{R,r}(x_0,t_0))$, we will use the notation
\[
\overline{u}_{(x_0,t_0),R,r} = \dashint_{Q_{R,r}(x_0,t_0)}\,u\,\dd x\, \dd t.
\]
When the center $(x_0,t_0)$ is clear from the context, we simply write $\overline{u}_{R,r}$. 
   
    \begin{proposition}\label{prop:timereg}
        Let $1<p<2,\,0<s<1$ and $\Theta(s,p)=\min(\frac{sp}{p-1},1)$. Suppose that $u$ is a local weak solution of 
$$
\partial_t u+(-\Delta_p)^s u = f\text{ in } B_2\times (-2^{sp},0]
$$
as in Definition \ref{subsupsolution},  where 
$$
\norm{u}_{L^{\infty}(B_1 \times{(-1,0]})} \leq 1, \quad \sup_{t \in {(-1,0]}} \mathrm{Tail}_{p-1,s\,p}(u(\cdot,t);0,1) \leq 1,  \quad \text{and} \quad \norm{f}_{L^\infty(B_1 \times (-1,0])} \leq 1.
$$
 Further, assume that for some $\delta\in(s,\Theta(s,p))$, there exists a positive constant $K_\delta$ depending on $\delta$ such that
\begin{equation}
\label{eq:holderseminorm-condit}
\sup_{t\in {(-2^{-sp},0]}}[u(\cdot,t)]_{C^\delta(B_{1/2})}\leq K_\delta.
\end{equation}
Then there is a constant $C= C(N,s,p,K_\delta,\delta)>0$ such that 
$$
\abs{u(x,t)-u(x,\tau)} \leq C\, \abs{t-\tau}^{\gamma},\qquad \text{ for all } x\in B_\frac{1}{4}\text{ and for all } t,\tau \in (-4^{-sp},0] 
$$
where 
$$
\gamma = \frac{1}{\dfrac{s\,p}{\delta}+(2-p)}. 
$$
In particular, $u\in C^\gamma_{t}(Q_{\frac{1}{4},4^{-sp}})$.
    \end{proposition}
    \begin{proof}
   Take $(x_0,t_0)\in Q_{\frac{1}{4},4^{-sp}}$ and choose 
\[
0<r<\frac{1}{8},\qquad 0<\theta<\frac{1}{2}\left(2^{-sp} - 4^{-sp} \right).
\]
By construction,
\[
Q_{r,\theta}(x_0,t_0)\subset B_\frac{3}{8}\times \left(-2^{-sp},0\right].
\]
Let $\eta\in C_c^\infty(B_{r/2}(x_0))$ be a non-negative smooth function, such that
\[
\eta\equiv \|\eta\|_{L^\infty(B_{r/2}(x_0))} \mbox{ on } B_{r/4}(x_0),\qquad \overline{\eta}_{r}={\frac{1}{|B_r(x_0)|}\int_{B_r(x_0)}\eta\,\dd x}=1\,\,\qquad \mbox{ and }\qquad \|\nabla \eta\|_{L^\infty(B_{r/2}(x_0))}\le \frac{C}{r},
\] 
for some constant $C=C(\|\eta\|_{L^\infty(B_{r/2}(x_0))},N)>0$. Note that since the average of $\eta$ is $1$, we have the estimate
\[
\|\eta\|_{L^\infty(B_{r/2}(x_0))}=\frac{1}{|B_{r/4}(x_0)|}\,\int_{B_{r/4}(x_0)} \eta\,\dd x\le \frac{|B_{r}(x_0)|}{|B_{r/4}(x_0)|}\,\overline{\eta}_r=4^N.
\]
By the triangle inequality,
\begin{equation}\label{eq:Ai}
\begin{aligned}
\dashint_{Q_{r,\theta}(x_0,t_0)}|u(x,t)-\overline{u}_{r,\theta}|\,\dd x\, \dd t&\leq \dashint_{Q_{r,\theta}(x_0,t_0)}\left|u(x,t)-\overline{(u\,\eta)}_r(t)\right|\,\dd x\,\dd t\\
&+\dashint_{Q_{r,\theta}(x_0,t_0)}\left|\overline{u}_{r,\theta}-\overline{(u\,\eta)}_{r,\theta}\right|\,\dd x\,\dd t \\
&+ \dashint_{Q_{r,\theta}(x_0,t_0)}\left|\overline{(u\,\eta)}_{r,\theta}-\overline{(u\,\eta)}_r(t)\right|\,\dd x\,\dd t\\
& =: A_1+A_2+A_3,
\end{aligned}
\end{equation}
where
\[
\overline{(u\,\eta)}_{r}(t) = \dashint_{B_r(x_0)}u(y,t)\,\eta(y)\, \dd y.
\]
Following the same steps in the proof of \cite[Proposision 6.2]{BLS} we arrive at
\begin{equation}\label{eq:time-reg-A2}
A_2 \leq A_1+A_3, 
\end{equation}
\begin{equation}\label{eq:time-reg-A1}
\begin{split}
A_1\le C\,K_\delta\, r^{\delta}, \qquad  \mbox{ for some } C=C(N,s,p)>0,
\end{split}
\end{equation}
and 
\begin{equation}
\label{eq:time-reg-A3b}
A_3\le \sup_{T_0,T_1\in (t_0-\theta,t_0]}\left|\overline{(u\,\eta)}_r(T_0) - \overline{(u\,\eta)}_r(T_1)\right|. 
\end{equation}
For $T_0,T_1 \in (t_0-\theta,t_0]$ with $T_0<T_1$, we use the weak formulation \eqref{wksol} with $\phi(x,t)=\eta(x)$, to obtain
\begin{equation}
\label{eq:time-reg-A3bb}
\begin{split}
|B_r(x_0)|\,\Big|\overline{(u\,\eta)}_r(T_0) &- \overline{(u\,\eta)}_r(T_1)\Big| = \left|\int_{B_r(x_0)}u(x,T_0)\,\eta(x)\, \dd x - \int_{B_r(x_0)}u(x,T_1)\,\eta(x)\, \dd x\right|\\
&=\left|\int_{T_0}^{T_1}\iint_{\mathbb{R}^N\times\mathbb{R}^N}J_p(u(x,\tau)-u(y,\tau))\,(\eta(x)-\eta(y))\dd \mu(x,y)\dd \tau\right|\\
& \leq \left|\int_{T_0}^{T_1}\iint_{B_r(x_0)\times B_r(x_0)}J_p(u(x,\tau)-u(y,\tau))\,(\eta(x)-\eta(y)) \dd \mu(x,y) \dd\tau\right|\\
&+ 2\,\left|\int_{T_0}^{T_1}\iint_{(\mathbb{R}^N\setminus B_r(x_0))\times B_{r/2}(x_0)}J_p(u(x,\tau)-u(y,\tau))\,\eta(x)\dd \mu(x,y)\dd \tau \right|\\
& + \left| \int_{T_0}^{T_1} \int_{B_r(x_0)} f(x,\tau)\eta(x) \dd x \dd \tau \right| \\
& = J_1 + J_2 + J_3. 
\end{split}
\end{equation}
For $J_1$, the estimate performed in the proof of \cite[Proposition 6.2]{BLS} still applies and gives
\begin{equation}\label{eq:time-reg-J1}
   J_1\leq C K_{\delta}^{p-1} \theta r^{N+ \delta(p-1)-sp},
\end{equation}
for some constant $C=C(N,s,p,\delta)>0$. Regarding $J_2$, we may follow the steps in the proof of  \cite[Proposition 5.7]{T} and obtain
\begin{equation}\label{eq:time-reg-J2}
    J_2 \leq C \theta r^N \left( 1 + r^{\delta(p-1) -sp}\right) \leq C\theta r^{N+\delta(p-1)-sp},
\end{equation}
for some constant $C=C(N,s,p,\delta,K_{\delta})>0$.
Finally, we estimate $J_3$ as follows
\begin{equation}\label{eq:time-reg-J3}
\begin{aligned}
 J_3 &\leq \int_{T_0}^{T_1} \int_{B_r(x_0)} \abs{f(x,\tau)}\eta(x) \dd x \dd t \leq \norm{f}_{L^\infty(B_r(x_0)\times(T_0,T_1))} \int_{T_0}^{T_1} \int_{B_r(x_0)} \eta(x) \dd x \dd t\\
 & \leq (T_1-T_0) \abs{B_r(x_0)}  \,\norm{f}_{L^\infty(B_1\times(-1,0))} \dashint_{B_r(x_0)} \eta(x) \dd x \leq C(N) \theta r^{N} .
 \end{aligned}
\end{equation}
The combination of \eqref{eq:time-reg-A3b}--\eqref{eq:time-reg-J3} implies
\begin{equation}\label{eq:time-reg-A3c}
    A_3 \leq C(1+K_\delta^{p-1}) \theta(1+ r^{\delta(p-1) -sp}) \leq C(1+K_\delta^{p-1}) \theta r^{\delta(p-1) -sp}.
\end{equation}
Therefore, by \eqref{eq:Ai} together with \eqref{eq:time-reg-A2}
\begin{equation}\label{eq:time-req-campanato}
\dashint_{Q_{r,\theta}(x_0,t_0)}|u(x,t)-\overline{u}_{r,\theta}|\,\dd x\,\dd t \leq C\left( r^\delta + \theta r^{\delta(p-1) -sp} \right),
\end{equation}
for some constant $C=C(N,s,p,\delta,K_{\delta})>0$.\\
We now split the rest of the proof into two cases.\\
\textbf{Case 1:} $sp +\delta(2-p) \geq 1$. Choose
$$\theta = \frac{1}{2}\left(\frac{1}{2^{sp}} - \frac{1}{4^{sp}}\right)r^{sp+\delta(2-p)} .$$
Then \eqref{eq:time-req-campanato} reads
\[
\dashint_{Q_{r,\theta}(x_0,t_0)} \abs{u(x,t)-\Bar{u}_{r,\theta} } \dd x \dd t \leq Cr^{\delta},
\]
for some constant $C=C(N,s,p,\delta,K_{\delta})>0$. \\
Now we use the characterization of the Campanato spaces in $\R^{n+1}$ with a general metric in \cite{Gorka}, see also \cite{dapra}. Our setting does not fit directly in the context considered there, since we only work with cylinders that are one sided in the time direction, that is $(t-r^{sp+\delta(2-p)},t]\times B_r(x)$ instead of $(t-r^{sp+\delta(2-p)},t+r^{sp+\delta(2-p)})\times B_r(x)$. Still, if {one follows} the proof in \cite{Gorka} with small modifications, {one} can also conclude the result in this setting. By \cite[Theorem 3.2]{Gorka}, $u$ is $\delta$-H\"older continuous in $Q_{1/4,1/4^{sp}}$ with respect to the metric
$$d((x,\tau_1),(y,\tau_2))= \max{\lbrace\abs{x-y} , \abs{\tau_2 - \tau_1}^{\frac{1}{sp+\delta(2-p)}}\rbrace}.$$
Note that since $sp+ \delta(2-p)\geq 1$, $d$ is a metric. The balls of radius $r$ for this metric are of the form $(t-r^{sp+\delta(2-p)},t+r^{sp+\delta(2-p)})\times B_r(x)$.

In particular, for any $\tau_1 , \tau_2 \in (-\frac{1}{4^{sp}} , 0]$ we have the estimate
\[
\sup_{x \in \overline{B_{1/4}}} \abs{u(x,\tau_1)-u(x,\tau_2)} \leq C \abs{\tau_1-\tau_2}^{\gamma},
\]
where the constant $C= C(\delta, K_\delta ,N, s, p)>0$, and
$$\gamma = \frac{\delta}{sp + \delta(2-p)} = \frac{1}{\frac{sp}{\delta} + 2-p}.$$
\textbf{Case 2:} $sp +\delta(2-p)<1$. In this case, we make the choice
$$r = \frac{1}{2}\left(\frac{1}{2^{sp}} - \frac{1}{4^{sp}} \right) \theta^{\frac{1}{sp + \delta(2-p)}}.$$
From \eqref{eq:time-req-campanato} we obtain
\[
\dashint_{Q_{r,\theta}(x_0,t_0)} \abs{u(x,t)-\Bar{u}_{r,\theta} } \dd x \dd t \leq C \theta^{\frac{\delta}{sp+\delta(2-p)}}.
\]
Again, by \cite[Theorem 3.2]{Gorka} we obtain that $u$ is {$\frac{\delta}{sp+\delta(2-p)}$-H\"older continuous}, in $Q_{1/4,1/4^{sp}}$ with respect to the metric 
$$\tilde{d}=\max{\left\lbrace \abs{x-y}^{sp+\delta(2-p)}, \abs{\tau_1-\tau_2}   \right\rbrace}.$$
As $sp+\delta(2-p)<1$, {$\tilde{d}$} is indeed a metric. In particular, we have the estimate
\[
\sup_{x \in \overline{B_{1/4}}} \abs{u(x,\tau_1)-u(x,\tau_2)} \leq C \abs{\tau_1-\tau_2}^{\gamma},
\]
for some constant $C= C(\delta, K_\delta ,N, s, p)>0$, where 
$$\gamma = \frac{\delta}{sp+\delta(2-p)}.$$
    \end{proof}

\section{Proof of the main theorem}
We are now ready to give the proof of our main theorem, namely Theorem \ref{teo:1}, that encompasses both the spatial regularity and regularity in time.

\begin{proof}[~Proof of Theorem \ref{teo:1}] The proof follows from a combination of {Theorem \ref{teo:new}}, Remark \ref{rem:covering} and Proposition \ref{prop:timereg}. We spell out the details below.
It remains to prove the regularity in time. We  assume $x_0=0$ and $T_0=0$ and argue as in the proof of {Theorem \ref{teo:new}}. Let 
\[
u_R(x,t):=\frac{1}{M}\,u(RM^\frac{p-2}{sp} x+y_0,R^{sp}t),\qquad \mbox{ for }(x,t)\in B_2\times (-2^{sp},0]
\]
where $y_0$ is chosen so that 
\begin{equation}
\label{x0ass2}
\frac12 RM^\frac{p-2}{sp}+|y_0|\leq  R/2
\end{equation}
and
\[
\begin{aligned}
M=&\left(2+\left(\frac{2 N \omega(N)}{sp}\right)^\frac{1}{p-1}\right)\|u\|_{L^\infty(B_{R}\times (-R^{sp},0])} 
{ +2^{1+N+sp}\sup_{t\in (-R^{sp},0]}\mathrm{Tail}_{p-1,s\,p}(u(\cdot,t);0,R)}^{p-1}\\
&+ R^{sp}\norm{f}_{L^\infty(B_R\times(-R^{sp},0])} +1.
\end{aligned}
\]
Estimate \eqref{x0ass2} implies
$$
2RM^\frac{p-2}{sp}+|y_0|\leq 2R.
$$
Therefore, $u_R$ is a local weak solution of $\partial_t u+(-\Delta_p)^s u= \tilde{f}$ in $B_2\times (-2^{sp},0]$, where
$$\tilde{f}(x,t):= \frac{R^{sp}}{M} f(RM^{\frac{p-2}{sp}}x +y_0 , R^{sp}t).$$
With the notation $\tilde R=RM^\frac{p-2}{sp}$, it is straight forward to verify that \eqref{x0ass2} implies
$$
\|u_R\|_{L^\infty(B_1\times (-1,0])}=M^{-1}\|u\|_{L^\infty(B_{\tilde R}(y_0)\times (-R^{sp},0])}\leq M^{-1}\|u\|_{L^\infty(B_{R}\times (-R^{sp},0])}\leq 1.
$$
As in the proof of Theorem \ref{teo:new} we also have
$$
\sup_{t\in (-1,0]}\int_{\mathbb{R}^N\setminus B_1}\frac{|u_R(z,t)|^{p-1}}{|{z}|^{N+s\,p}}\,  dz\leq 1.
$$
We also have
\[
\begin{aligned}
\norm{\tilde{f}}_{L^\infty(B_1\times(-1,0])} &=\frac{R^{sp}}{M}
 \left\|  f(RM^{\frac{p-2}{sp}}x+ y_0, R^{sp}t)  \right\|_{L^\infty(B_1\times (-1,0])} \\ 
 &= \frac{R^{sp}}{M} \norm{f}_{L^\infty\left(B_{\tilde R}(y_0)\times(-R^{sp},0]\right)} \leq \frac{R^{sp}}{M} \norm{f}_{L^\infty(B_R\times (-R^{sp},0])} \leq 1.
 \end{aligned}
\]
The above bounds combined with {Theorem \ref{teo:new}} and Remark \ref{rem:covering} imply
\begin{equation}\label{eq:finalRxreg}
\sup_{t\in {\left(-2^{-sp},0\right]}} [u_R(\cdot,t)]_{C^{\theta}(B_{\frac12})}\leq C,
\end{equation}
for every $\theta\in(0,\Theta)$ with $C=C(N,s,p,\theta)>0$. By using this together with Proposition \ref{prop:timereg}, we obtain
\begin{equation}\label{eq:finalRtreg}
\sup_{x\in B_\frac14}[u_R(x,\cdot)]_{C^\gamma((-4^{-sp},0])}\leq C
\end{equation}
where 
$$
\gamma = \frac{1}{\dfrac{s\,p}{\theta}+(2-p)}, 
$$ and $C=C(N,s,p,\theta)$. Scaling back, \eqref{eq:finalRxreg} and \eqref{eq:finalRtreg} imply, (again with the notation $\tilde R=RM^\frac{p-2}{sp}$)
\begin{equation}\label{1}
\sup_{t\in (-(R/2)^{sp},0]}[u(x,t)]_{C^{\theta}(B_{\frac{\tilde R}{2}}(y_0))}\leq CM^{1+\frac{(2-p)}{sp}}R^{-\theta}
\end{equation}
and
\begin{equation}\label{2}
\sup_{x\in B_\frac{\tilde R}{4}(y_0) } [u(x,\cdot)]_{C^\gamma((-4^{-sp}R^{sp},0])}\leq CMR^{-\gamma}
\end{equation}
for all $y_0$ that satisfies \eqref{x0ass2}, with the constant $C=C(N,s,p,\theta)>0$. Therefore, we obtain the estimates \eqref{1} and \eqref{2} in the whole $B_{R/4}$ by varying $y_0$ as in the proof of {Theorem \ref{teo:new}}. Since this holds for all $\theta<\Theta$ and with $\gamma = ({s\,p}/{\theta}+(2-p))^{-1}$ {which converges to $\Gamma$} as $\theta\to \Theta$, this implies the desired result.
\end{proof}

\section*{Acknowledgement} The first author is supported by the seed grant from IISER Berhampur. E.L and A.T have been supported by the Swedish Research Council, grants
no. 2023-03471 and 2016-03639.
\bibliographystyle{plain}
\bibliography{bibfile}

\vskip 0.1cm
Prashanta Garain\\
Department of Mathematical Sciences\\
Indian Institute of Science Education and Research Berhampur\\
Berhampur, Odisha 760010, India\\
\emph{E-mail address:} \email{pgarain92@gmail.com} \\

Erik Lindgren\\
Department of Mathematics\\
KTH - Royal Institute of Technology
100 44, Stockholm, Sweden\\
\emph{E-mail address:} \email{eriklin@math.kth.se} \\

Alireza Tavakoli\\
Department of Mathematics\\
KTH - Royal Institute of Technology
100 44, Stockholm, Sweden\\
\emph{E-mail address:} \email{alirezat@kth.se} \\

\end{document}